\documentclass[11pt,letterpaper]{amsart}
\usepackage[lmargin=1.25in,rmargin=1.25in,tmargin=1in,bmargin=1in]{geometry}

\usepackage{
	amsmath,			
	amssymb,			
    amsthm,
    bm,
	enumerate,		
    float,
	graphicx,			
    arydshln,
    thmtools,
	lastpage,			
	multicol,			
	multirow,			
    parskip,
    pdfpages,
    setspace,
    tikz,
    tikzscale,
    thm-restate,
    verbatim,
	xcolor
}

\usetikzlibrary{
	arrows,
    arrows.meta,
    calc,
    decorations.markings,
	decorations.pathmorphing,
	matrix,
	positioning,
    shapes.arrows,
	shapes.geometric
}
\usepackage{pgfplots}
\pgfplotsset{compat=newest}

\usepackage[pdfencoding=auto, psdextra]{hyperref}
\usepackage{cleveref}

\hypersetup{
    colorlinks=true,
    linkcolor=blue,
    filecolor=magenta, 
    citecolor=cyan,
}

\usepackage[T1]{fontenc}
\usepackage{charter}

\newtheorem{theorem}{Theorem}[section]
\newtheorem{lemma}[theorem]{Lemma}
\newtheorem{prop}[theorem]{Proposition}

\theoremstyle{definition}

\theoremstyle{remark}
\newtheorem{remark}[theorem]{Remark}

\theoremstyle{remark}

\theoremstyle{remark}

\makeatletter
\newcommand{\addresseshere}{%
  \enddoc@text\let\enddoc@text\relax
}
\makeatother

\usepackage{fancyhdr}

\newcommand{\RP}{\mathbb{RP}}

\newcommand{\ZZ}{\mathbb{Z}}
\newcommand{\K}{K_{3,3}}

\DeclareMathOperator{\fw}{fw}
\DeclareMathOperator{\girth}{g}

\begin{document}

\title{Cubic torus obstructions of small Betti number}
\author{Marie Kramer}
\address{Department of Mathematics, Syracuse University, Syracuse, NY USA}
\email{mkrame04@syr.edu}
\begin{abstract}
    The embeddability of graphs into surfaces has been studied for nearly a century. While the complete set of topological obstructions is known for the sphere and the real projective plane, there are only partial results for the torus. Here we present a theoretical classification of cubic torus obstructions with Betti number at most eight.
\end{abstract}

\maketitle

\section*{Introduction}
\label{sec:intro}

Kuratowski's Theorem states that a graph is planar if and only if it does not contain $K_5$ or $K_{3,3}$ as a topological minor, that is, a subgraph isomorphic to a subdivision $K_5$ or $\K$ \cite{Kuratowski30}. We say $K_5$ and $\K$ are the topological obstructions for the plane. 
Via stereographic projection, embeddability into the plane is equivalent to embeddability into the two-dimensional sphere. This raises the question of how to find lists of obstructions for other compact surfaces. Robertson and Seymour proved that these lists are finite for every surface \cite{RobertsonSeymour90}.
The only other surface besides the sphere for which the complete set of topological obstructions is known is the real projective plane. Glover, Huneke, and Wang constructed 103 topological obstructions \cite{GloverHunekeWang79}, six of which are cubic \cite{GloverHuneke75}. Archdeacon then showed that this list is indeed complete \cite{Archdeacon81}.

The natural next surface to consider is the torus. There are over 250,000 known topological obstructions but this list is believed to be incomplete \cite{MyrvoldWoodcock18}. 
While there are some theoretical results regarding disconnected obstructions and obstructions of vertex-connectivity one, most of the work has been done using computers. Chambers found 206 cubic topological torus obstructions with the aid of computers, and it is believed that this list is complete \cite{Chambers02}. The largest of these are connected with 24 vertices, 36 edges, and Betti number 13. We refer to \cite{MyrvoldWoodcock18} for a more comprehensive overview of the progress towards a complete list of topological torus obstructions.

Our main theorem provides a theoretical proof that the list in \cite{Chambers02} of cubic torus obstructions is complete up to Betti number eight. It moreover provides insight into the structure of these graphs.

\begin{restatable}{thmx}{TorusVersion}
    \label{thm:CubicTorusObstructions}
    Let $G$ be a cubic graph with Betti number at most eight.
    If $G$ does not embed into the torus, then $G$ is one of eleven pairwise non-isomorphic graphs $H_0, H_1, \ldots, H_9, H_0 \cup e$. Moreover, the graphs $H_0, H_1, \ldots, H_9$ shown below are cubic obstructions for the torus.
\end{restatable}
\begin{figure}[H]
    \centering
    \includegraphics[height=1.9cm]{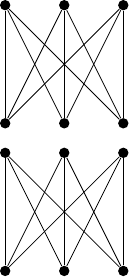}
    \includegraphics[height=1.9cm]{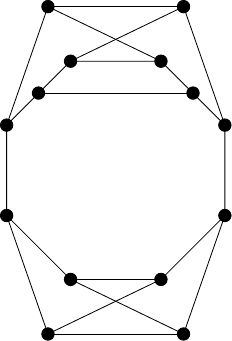}
    \includegraphics[height=1.9cm]{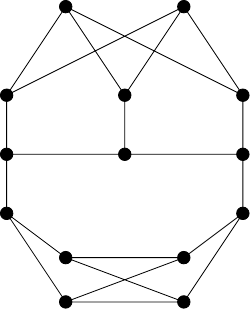}
    \includegraphics[height=1.9cm]{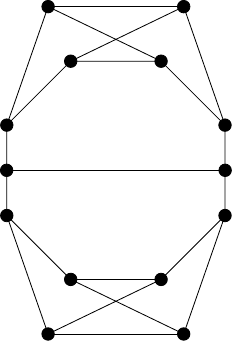}
    \includegraphics[height=1.9cm]{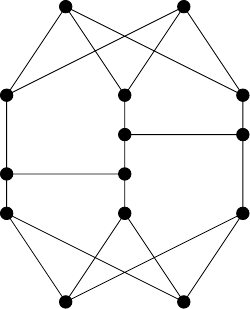}
    \includegraphics[height=1.9cm]{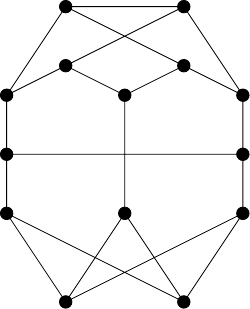}
    \includegraphics[height=1.9cm]{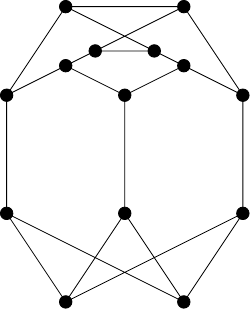}
    \includegraphics[height=1.9cm]{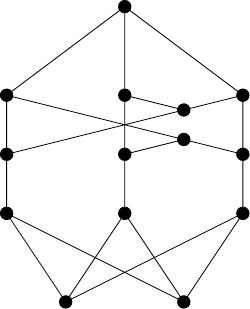}
    \includegraphics[height=1.9cm]{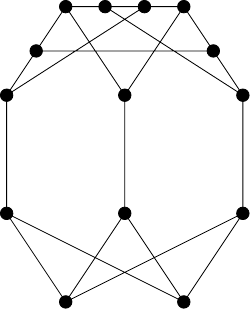}
    \includegraphics[height=1.9cm]{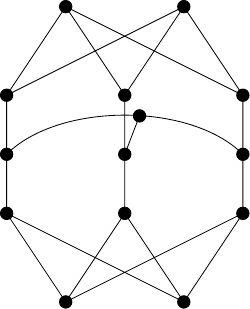}
\end{figure}
Previous work \cite{KrakovskiMohar14} implies that projective planar cubic graphs with sufficiently small Betti number have small facewidth, also called representativity. In particular, this allows for a projective planar embedding to be transformed into a toroidal one \cite{FiedlerHunekeRichterRobertson95}.
Therefore, to prove Theorem \ref{thm:CubicTorusObstructions}, we use the six cubic projective plane obstructions to build candidates for cubic torus obstructions. A key tool in determining which of the candidates are obstructions is the classification of toroidal embeddings of cubic projective plane obstructions from \cite{Kramer24}.

We restrict our attention to \textit{cubic graphs} for two reasons. First, during the classification of projective plane obstructions the cubic case was considered \cite{GloverHuneke75} before moving on to the case of bounded degree \cite{GloverHuneke77} and finally the general case \cite{GloverHunekeWang79}, \cite{Archdeacon81}.
Second, the restriction to cubic graphs is motivated by a geometric application of our main theorem discovered by Kennard, Wiemeler, and Wilking \cite{KennardWiemelerWilking22}. They showed that if a graph embeds into a surface, one obtains strong upper bounds on its systole. It turns out that given any graph $G$, there exists a cubic graph $G'$ with the same Betti number such that the systole of $G$ is bounded above by that of $G'$. Thus, for finding these systole bounds, cubic graphs represent extreme cases, and the proof may be reduced to the cubic case.

The second restriction is to only consider graphs of \textit{small Betti number}. This is also motivated by results from \cite{KennardWiemelerWilking22}. Systole bounds of graphs give bounds on matroid theoretic invariants, which in turn give geometric bounds for special tours representations of small rank. The rank of the torus representation equals the Betti number of the graph in the important special case of cographic representations, and therefore Betti numbers in the range of Theorem \ref{thm:CubicTorusObstructions} are of interest. These geometric bounds have seen applications in Riemannian geometry in \cite{KennardWiemelerWilking22} (see also \cite{KennardWiemelerWilking21}).

\textbf{Acknowledgements.} This work is part of the author's Ph.D. thesis, and she would like to thank her advisor Lee Kennard. The author is grateful for support from NSF Research Grants DMS-2005280 and DMS-2402129, as well as from the Syracuse University Graduate School through Pre-Dissertation and Dissertation Fellowships in Summer 2023 and Summer 2024, respectively.
% ------------------------------------------------------------------------------------------------
% ------------------------------------------------------------------------------------------------

% ------------------------------------------------------------------------------------------------
% ------------------------------------------------------------------------------------------------
\section{Preliminaries}\label{sec:prelim}
Unless otherwise specified, a graph $G=(V,E)$ is assumed to be undirected and simple.
Most of the graphs we work with are \textbf{\textit{cubic}}, that is, they only have vertices of degree three. The \textbf{\textit{girth}} of $G$, denoted by $\girth(G)$, is the length of its shortest cycle.

We say a graph $G$ is \textbf{\textit{$\boldsymbol{k}$-connected}} if it is $k$-edge-connected in the sense that $|V(G)| \geq 2$ and any cut set of edges has at least $k$ elements. A graph has \textbf{\textit{connectivity $\boldsymbol{k}$}} if it is $k$-connected and there exists a cut set with $k$ elements. As we are working with cubic graphs, the notion of cyclically $4$-connectedness is useful. A cubic graph is \textbf{\textit{cyclically $\boldsymbol{4}$-connected}} if it is $3$-connected and the removal of any cut set of size three leaves at most one component containing a cycle. The figure below shows two $3$-connected cubic graphs, where only the second one is cyclically $4$-connected. We refer to \cite{KinganKingan24} for more information on cyclically $4$-connected cubic graphs.
\begin{figure}[H]
    \centering
    \includegraphics[height=2.5cm]{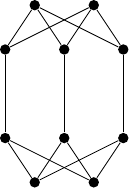} \hspace{1cm}
    \includegraphics[height=2.5cm]{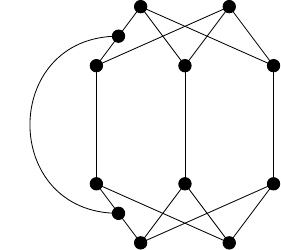}
\end{figure}

The \textbf{\textit{Betti number}} $b_1$ of a graph is defined in the topological sense, that is, as the rank of the first homology group of the graph when viewed as a CW-complex. This invariant equals the cyclomatic number and the dimension of the cycle basis. A short calculation yields \[b_1 = |E|-|V|+b_0,\] where $b_0$ is the rank of the 0$^{\text{th}}$ homology group of the graph or equivalently the number of components of the graph. 
For cubic graphs, we have $2|E|=3|V|$ by the Handshaking Lemma, and thus \[b_1=\dfrac{1}{2}|V|+b_0.\]

Two graphs $G$ and $G'$ are \textbf{\textit{isomorphic}} if there exists a one-to-one correspondence between their vertices such that a pair of vertices is adjacent in $G$ if and only if the corresponding pair of vertices is adjacent in $G'$. If $G$ and $G'$ are isomorphic, we denote this by $G \cong G'$.
Two graphs $G$ and $G'$ are \textbf{\textit{homeomorphic}} if some subdivision of $G$ is isomorphic to some subdivision of $G'$. If $G$ and $G'$ are homeomorphic, we denote this by $G \simeq G'$.

An \textbf{\textit{embedding}} of a graph $G$ into a surface $\Sigma$ is a function $\varphi$ mapping the vertices of $G$ to points in $\Sigma$ and the edges of $G$ to continuous curves in $\Sigma$ such that curves representing distinct edges intersect only at images of vertices (see \textsection 15.1 in \cite{GraphsAlgosAndOpti}). We denote an embedding by $\varphi: G \hookrightarrow \Sigma$.
A \textbf{face} of an embedding is a connected component of $\Sigma \setminus \varphi(G)$.
A \textbf{\textit{facial cycle}} or \textbf{\textit{facial walk}} is an oriented cycle or walk in $G$ along the boundary of a face of $\varphi(G)$.
An embedding of $G$ into $\Sigma$ is \textbf{\textit{cellular}} if each face is homeomorphic to an open disk. Note that all embeddings into the sphere are cellular.

Two embeddings $\varphi_1$ and $\varphi_2$ of a graph $G$ into a surface $\Sigma$ are \textbf{\textit{equivalent}} if there exists a homeomorphism $h: \Sigma \to \Sigma$ such that the images of $G$ under $\varphi_1$ and $h \circ \varphi_2$ agree (see \textsection 15.2 in \cite{GraphsAlgosAndOpti}). Since any homeomorphism of a surface maps facial walks to facial walks, two embeddings with facial walks of different lengths cannot be equivalent.

A graph is called \textbf{\textit{planar}} if it embeds into the sphere. Similarly, a graph is called \textbf{\textit{projective planar}} if it embeds into the real projective plane, and it is called \textbf{\textit{toroidal}} if it embeds into the torus. Via stereographic projection it can be show that a graph embeds into the sphere if and only if it embeds into the plane (see \textsection 1.3 in \cite{PlanarGraphs}). 

A graph $G$ is a \textbf{\textit{topological obstruction}} for a surface $\Sigma$ if $G$ does not embed into $\Sigma$ but every topological minor does. Here, a \textbf{\textit{topological minor}} is a graph obtained from $G$ by deleting edges and/or vertices and contracting edges with an endpoint of degree two.
A graph is a \textbf{\textit{cubic obstruction}} for a surface $\Sigma$ if it is cubic and a topological obstruction for $\Sigma$.

Let $G$ be a projective planar graph, and $\varphi: G \hookrightarrow \RP^2$ one such embedding.
The \textbf{\textit{facewidth of the embedding $\boldsymbol{\varphi}$}}, denoted $\fw(\varphi)$, is the minimum number of intersection points of $\varphi(G)$ with any essential cycle in $\RP^2$.
The \textbf{\textit{facewidth of the graph $\boldsymbol{G}$}}, denoted $\fw(G)$, is the minimum facewidth over all embeddings of $G$ into $\RP^2$.

We record here an elementary observation about cubic obstructions. In particular, it shows that cubic obstructions are simple even if it is not explicitly assumed.

\begin{prop}
    \label{prop:ObstructionGirth}
    If $G$ is a cubic obstruction for a surface $\Sigma$, then $G$ has girth at least four.
\end{prop}
\begin{proof}
    Let $G$ be a cubic obstruction for $\Sigma$.
    If $e$ is an edge that is either a loop or parallel to another edge, then $G \setminus e$ embeds into $\Sigma$ by assumption. It is easy to see that any such embedding extends to an embedding of $G$. Hence, $G$ does not have girth one or two.
    Lastly, suppose $\girth(G)=3$, that is, $G$ has girth three. Let $e_1, e_2, e_3 \in E(G)$ be a 3-cycle in $G$. By assumption, $G \setminus e_3$ embeds into $\Sigma$. Note that as $G$ is cubic, any embedding of $G \setminus e_3$ has a face $f$ with both $e_1$ and $e_2$ on its boundary. Thus, we can extend any embedding of $G \setminus e_3$ to an embedding of $G$ by embedding $e_3$ into the face $f$, contradicting $G$ being an obstruction.
\end{proof}
% ------------------------------------------------------------------------------------------------
% ------------------------------------------------------------------------------------------------

% ------------------------------------------------------------------------------------------------
% ------------------------------------------------------------------------------------------------
\section{Constructing cubic torus obstruction candidates}\label{sec:t2obstruction_candidates}
In this section, we outline how to construct graphs that are candidates for small cubic torus obstructions. In particular, we will see that any projective planar candidates are too large for our purposes.

\begin{prop}\label{prop:ProjectivePlanarToroidal}
    If $G$ is a cubic projective planar graph with Betti number at most eight, then $G$ is toroidal.
\end{prop}
\begin{proof}
    Let $G$ be a cubic projective planar graph with Betti number at most eight. It follows from \cite{FiedlerHunekeRichterRobertson95} that $G$ is toroidal whenever $\fw(G) \leq 3$. If $\fw(G) \geq 4$, then $G$ contains $K_6$ as a minor by \cite{KrakovskiMohar14}. Note that as $K_6$ is $5$-regular with six vertices, any cubic graph containing $K_6$ as a minor has at least 18 vertices. Hence, $b_1(G)=\frac{1}{2}|V(G)| + b_0(G) \geq 9 +1 = 10$, contradicting the assumption that $b_1(G) \leq 8$.
\end{proof}

We assume from now on that $G$ is a cubic graph of Betti number at most eight not embeddable into the projective plane.
By \cite{GloverHuneke75}, $G$ contains a subgraph homeomorphic to one of the six cubic projective plane obstructions $E_{42}, F_{11}, F_{12}, F_{13}, F_{14},$ and $G_1$ pictured below.
\begin{figure}[H]
    \centering
    \includegraphics[height=1.9cm]{images/E42.pdf} \hspace{.1cm}
    \includegraphics[height=1.9cm]{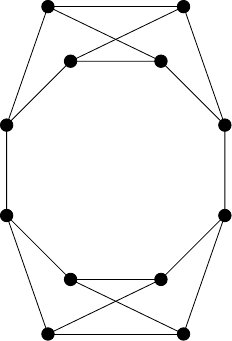} \hspace{.1cm}
    \includegraphics[height=1.9cm]{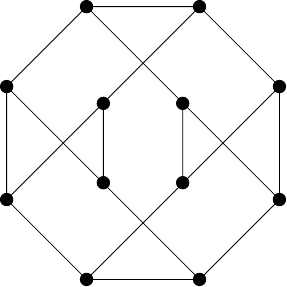} \hspace{.1cm}
    \includegraphics[height=1.9cm]{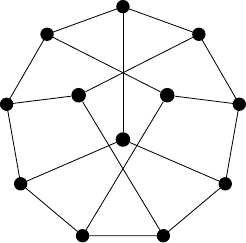} \hspace{.1cm}
    \includegraphics[height=1.9cm]{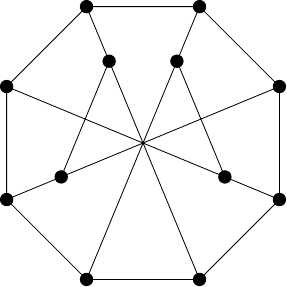} \hspace{.1cm}
    \includegraphics[height=1.9cm]{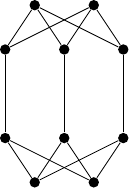} 
\end{figure}

The following operation will be used in the upcoming sections to build obstruction candidates.
Starting from a graph $G$, we obtain the graph $G \cup e$ by adding an edge in one of three ways. See Figure \ref{fig:ExcludingCases} for illustrations of these operations.
\begin{itemize}
    \item Let $e_1 \in E(G)$. Subdivide $e_1$ once, calling the new vertex $u$. Define the graph $G \cup e$ as this subdivision with additional vertex $v$ and additional edges $uv$ and $vv$.
    \item Let $e_1 \in E(G)$. Subdivide $e_1$ twice, calling the new vertices $u$ and $v$. Define the graph $G \cup e$ as this subdivision with the additional edge $uv$.
    \item Let $e_1, e_2 \in E(G)$ with $e_1 \neq e_2$. Subdivide each of $e_1, e_2$ once, calling the new vertices $u$ and $v$, respectively. Define the graph $G \cup e$ as this subdivision with the additional edge $uv$.
\end{itemize}

\begin{remark}\label{rmk:5AddingEdge} We collect some elementary observations about the effect of adding an edge to a graph as described above.
    \begin{enumerate}[(1)]
        \item If $G$ is cubic, so is $G \cup e$.
        \item By definition of $G\cup e$, $b_1(G \cup e)=b_1(G)+1$ if $G$ is connected or more generally if $b_0(G \cup e) = b_0(G)$, and $b_1(G \cup e)=b_1(G)$ otherwise.
        \item We build the candidates for cubic torus obstructions from cubic simple graphs and attach zero, one or two edges as described above. By Proposition \ref{prop:ObstructionGirth}, graphs of girth less than four cannot be cubic obstructions. Hence, if $G$ is toroidal and we only attach one edge $e$, we may assume that $\girth(G \cup e)\geq 4$. That is, we may exclude all cases pictured in Figure \ref{fig:ExcludingCases}.
        Similarly, if we attach two edges $e_1, e_2$, we may assume that $G \cup e_1$ has girth at least three. Indeed, if $\girth(G \cup e_1) \leq 2$, then $\girth(G \cup e_1 \cup e_2) \leq 3$. That is, we may exclude the first two of the cases pictured in Figure \ref{fig:ExcludingCases}.
            \begin{figure}[H]
                \centering
                \includegraphics[height=1.5cm]{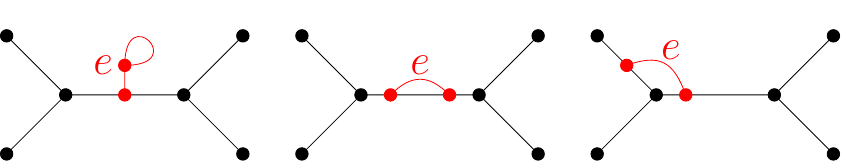}
                \caption{Cubic obstructions have girth at least four}
                \label{fig:ExcludingCases}
            \end{figure}
    \end{enumerate}
\end{remark}

Before introducing tools that will help us construct cubic torus obstructions from $F_{11}$, $F_{12}$, $F_{13}$, $F_{14}$, and $G_1$ in Sections \ref{sec:extendingF11-F14} and \ref{sec:extendingG1}, we consider the graph $E_{42}$.
By \cite{Kramer24}, $E_{42}$ does not embed into the torus. We moreover observe that $E_{42} \cong H_0$.
As $H_0$ is disconnected, Remark \ref{rmk:5AddingEdge} gives $b_1(H_0 \cup e)=b_1(H_0)=8$ if and only if $H_0 \cup e$ is connected. Moreover, edge-transitivity of $K_{3,3}$ implies that up to isomorphism, there exists exactly one connected graph $H_0 \cup e$.
As $H_0$ is not toroidal, neither is $H_0 \cup e$. Note that $H_0 \cup e$ cannot be an obstruction as its subgraph $H_0$ does not embed into the torus either.

When building candidates for torus obstructions from the remaining cubic projective plane obstructions, we use graph symmetries to avoid listing many redundancies up to graph isomorphisms. However, we do not completely classify the isomorphism classes as this would have been more time consuming than simply constructing additional embeddings as needed. Additionally, Remarks \ref{rmk:CubicYDelta} and \ref{rmk:5New4Cycles} below reduce the number of cases we need to consider in the upcoming sections.

\begin{remark}\label{rmk:CubicYDelta}
    Let $G$ be a cubic graph, and $G \cup e$ such that $e$ is part of a new $3$-cycle $C_3$ in $G \cup e$. It is easy to see that attaching $e$ in such a way is equivalent to a \textit{cubic $Y-\Delta$ transformation}, which we have illustrated below. In particular, the graph resulting from applying a cubic $Y-\Delta$ transformation to a cubic graph is cubic. Moreover, this equivalence implies that the isomorphism class of $G \cup e$ only depends on the orbit of the vertex $v \in V(G)$ with $v \in V(C_3)$ under the action of the automorphism group of $G$.
                    \begin{figure}[H]
                        \centering
                       \includegraphics[height=2cm]{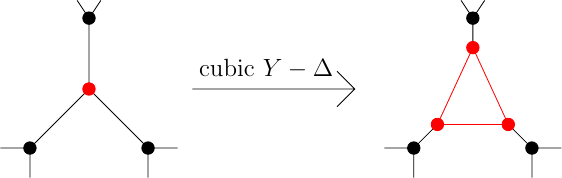}
                    \end{figure}
    While by Proposition \ref{prop:ObstructionGirth} $G \cup e$ cannot be an obstruction since $\girth(G \cup e) = 3$, we need to consider such graphs in cases where we attach a second edge. If one of the endpoints of the second edge is on an edge in $C_3$, this turns the $3$-cycle into a $4$-cycle and could therefore give a graph of girth at least four.
\end{remark}

\begin{remark}\label{rmk:5New4Cycles}
    Let $G$ be a cubic graph and $G \cup e$ such that $e$ is part of a new $4$-cycle $C_4$ in $G \cup e$. Denote the edge not incident to $e$ in $C_4$ by $e'$. By a result from \cite{KinganKingan24}, there are at most two isomorphism classes of graphs $G \cup e$ such that $C_4 = e h_1 e' h_2$, where $h_i \in E(G \cup e) \setminus (e \cup e')$ for $i \in \{1,2\}$.
    Therefore, after picking the edge $e' \in E(G)$ up to automorphism, there are at most two isomorphism classes of graphs $G \cup e$.
    \begin{figure}[H]
        \centering
        \includegraphics[height=1.5cm]{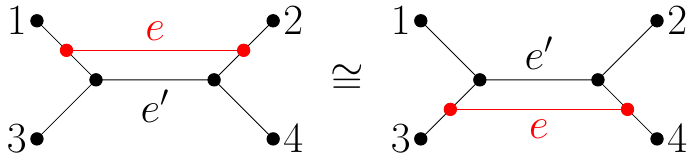} \hspace{.5cm}
        \includegraphics[height=1.5cm]{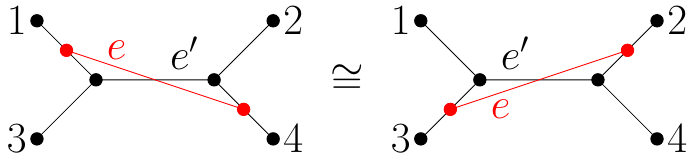}
    \end{figure}
\end{remark}
% ------------------------------------------------------------------------------------------------
% ------------------------------------------------------------------------------------------------

% ------------------------------------------------------------------------------------------------
% ------------------------------------------------------------------------------------------------
\section{Cubic torus obstruction candidates from \texorpdfstring{$F_{11}$}{F\_{11}}, \texorpdfstring{$F_{12}$}{F\_{12}}, \texorpdfstring{$F_{13}$}{F\_{13}}, and \texorpdfstring{$F_{14}$}{F\_{14}}}\label{sec:extendingF11-F14}
In this section, we construct candidates for cubic torus obstructions from the cubic projective plane obstructions $F_{11}, F_{12}, F_{13}$, and $F_{14}$. Let $i \in \{11, 12, 13, 14\}$. By \cite{Kramer24}, $F_i$ embeds into the torus. We moreover observe that  $b_1(F_i)=\frac{1}{2}|V(F_i)|+b_0(F_i)=6+1=7$ and $b_1(F_i \cup e)=8$ as $F_i$ is connected.

\begin{prop}
    \label{prop:F11Torus}
    $F_{11} \cup e$ is toroidal if $F_{11} \cup e \notin \{H_1, H_2, H_3, H_4\}$, where $H_1$ through $H_4$ are pictured below.
    \begin{figure}[H]
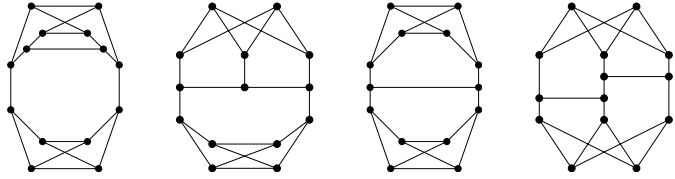

            \centering
            \includegraphics[height=2.25cm]{images/introduction/H1.pdf} \hspace{.5cm}
            \includegraphics[height=2.25cm]{images/introduction/H2.pdf} \hspace{.5cm}
            \includegraphics[height=2.25cm]{images/introduction/H3.pdf} \hspace{.5cm}
            \includegraphics[height=2.25cm]{images/introduction/H4.pdf}
            \caption{The graphs $H_1$ through $H_4$}
            \label{fig:H1ThroughH4}
        \end{figure}
\end{prop}
    
\begin{proof}
    We first observe that $F_{11}$ has three types of edges, where the equivalence classes represent the orbits under automorphisms.
        \begin{multicols}{2}
            \begin{itemize}
                \item $[AA']=\{AA', BB', A8, A'1, B5, B'4, 45, 18\}$
                \item $[A2]=\{A2, A'7, B3, B'6, 12, 34, 56, 78\}$
                \item $[23]=\{23, 67\}$
            \end{itemize}
            \newcolumn
            \begin{figure}[H]
                \centering
                \includegraphics[height=3.5cm]{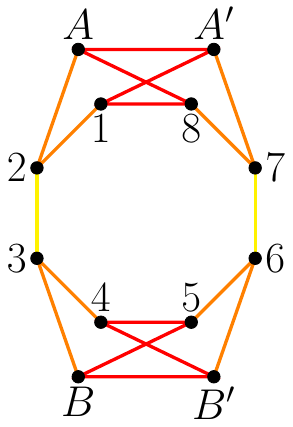}
                \label{fig:F11Orbits}
            \end{figure}
        \end{multicols}

    As $F_{11}$ is toroidal by \cite{Kramer24}, Proposition \ref{prop:ObstructionGirth} implies that $F_{11} \cup e$ is toroidal whenever \linebreak $\girth(F_{11} \cup e) \leq 3$. Up to isomorphism, there are at most ten ways to add one edge $e$ to $F_{11}$ such that $\girth(F_{11} \cup e) \geq 4$. We list these below. 
        \begin{itemize} 
            \item If one endpoint of $e$ attaches to an edge in $[AA']$, without loss of generality assume it attaches to $AA'$, we get the following five classes of edges in $E(F_{11})$ for the second endpoint of $e$ to attach to. Note that these equivalence classes are the orbits under the subgroup of automorphisms preserving the edge $AA'$. We always assume without loss of generality that we attach to the representative given in parenthesis.
                \begin{multicols}{2}
                    \begin{enumerate}[(1)]
                        \item $(BB')=\{BB', B5, B'4, 45\}$
                        \item $(B3)=\{B3, B'6, 34, 56\}$
                        \item $(12)=\{12, 78\}$
                        \newcolumn
                        \item $(23)=\{23, 67\}$
                        \item $(18)=\{18\}$
                    \end{enumerate}
                \end{multicols}
                
            \item If one endpoint of $e$ attaches to an edge in $[A2]$, without loss of generality assume it attaches to $A2$, and assuming the second endpoint does not attach to an edge in $[AA']$, we get the following four classes of edges in $E(F_{11})$ for the second endpoint of $e$ to attach to.
                \begin{multicols}{2}
                    \begin{enumerate}[(1)]
                        \setcounter{enumi}{5}
                        \item $(A'7)=\{A'7, 78\}$ % Same as AA'-18 by 4-cycle argument
                        \item $(B3)=\{B3, 34\}$
                        \item $(B'6)=\{B'6, 56\}$
                        \item $(67)=\{67\}$
                    \end{enumerate}
                \end{multicols}
                
            \item If one endpoint of $e$ attaches to an edge in $[23]$, without loss of generality assume it attaches to $23$, and assuming the second endpoint does not attach to an edge in neither $[AA']$ nor $[A2]$, we get the following class of edges in $E(F_{11})$ for the second endpoint of $e$ to attach to.
                \begin{multicols}{2}
                    \begin{enumerate}[(1)]
                        \setcounter{enumi}{9}
                        \item $(67)=\{67\}$
                    \end{enumerate}
                \end{multicols}
        \end{itemize}

    Using Remark \ref{rmk:5New4Cycles}, we observe that the graphs from Cases (5) and (6) are isomorphic.

    By \cite{Kramer24}, there are exactly two inequivalent unlabelled embeddings of $F_{11}$ into torus. It is easy to verify that for cases (1) - (4) and (7), the following labelled embedding of $F_{11}$ extends to an embedding of $F_{11} \cup e$.
        \begin{figure}[H]
                \centering
                \includegraphics[height=3.5cm]{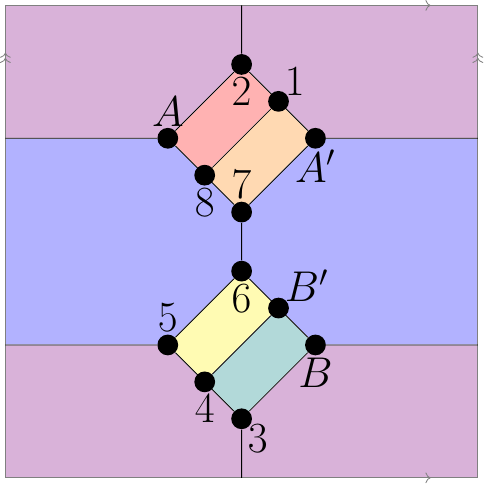}
            \end{figure} 

            Recalling the graphs $H_1$ through $H_4$ from Figure \ref{fig:H1ThroughH4}, we observe that
            \begin{multicols}{2}
                \begin{enumerate}[(1)]
                    \setcounter{enumi}{4}
                    \item $\cong F_{11} \cup (AA'-18) \cong H_1$,
                    \setcounter{enumi}{7}
                    \item $\cong F_{11} \cup (A2-B'6) \cong H_4$,
                    \newcolumn
                    \item $\cong F_{11} \cup (A2-67) \cong H_2$, and
                    \item $\cong F_{11} \cup (23-67) \cong H_3$.
                \end{enumerate}
            \end{multicols}
    \vspace{-1cm}
\end{proof}

\begin{prop}
    \label{prop:F12Torus}
    $F_{12} \cup e$ is toroidal for all choices of $e$.
\end{prop}

\begin{proof}
    We first observe that $F_{12}$ has three types of edges, where the equivalence classes represent the orbits under automorphisms:
        \begin{multicols}{2}
            \begin{itemize}
                \item $[AA']=\{AA', A3, A'2, BB', B6, B'7, 23, 67\}$
                \item $[A8]=\{A8, A'5, B1, B'4, 12, 34, 56, 78\}$
                \item $[45]=\{45, 81\}$
            \end{itemize}
            \newcolumn
            \begin{figure}[H]
                \centering
                \includegraphics[height=3.25cm]{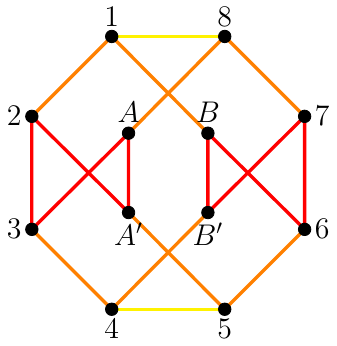}
            \end{figure}
        \end{multicols}

    Up to isomorphism, there are at most 14 ways to add one edge $e$ to $F_{12}$ such that $\girth(F_{12} \cup e) \geq 4$. We list these below. As they are obtained in a similar way as in the proof of Proposition \ref{prop:F11Torus}, we omit the details.
        \begin{multicols}{4}
            \begin{enumerate}[(1)]
                \item $AA'-BB'$
                \item $AA'-B1$
                \item $AA'-B6$
                \item $AA'-12$
                \item $AA'-23$
                \item $AA'-45$
                \item $AA'-56$
                \item $AA'-67$
                \item $A8-A'5$
                \item $A8-B1$
                \item $A8-B'4$
                \item $A8-12$
                \item $A8-45$
                \item $45-81$
                \newcolumn
            \end{enumerate}
        \end{multicols}

        Using Remark \ref{rmk:5New4Cycles}, we observe that the graphs from Cases (5) and (9) are isomorphic.

        By \cite{Kramer24}, there are exactly four inequivalent unlabelled embeddings of $F_{12}$ into the torus. For our purposes, the 6-cycle embedding and the 2-embedding suffice. Figure \ref{fig:F12Labelled} shows two labelled versions of the 6-cycle embedding, and one labelled version of the 2-embedding. For each of the 14 cases, it is easy to see that at least one of these labelled embeddings extends to an embedding of $F_{12} \cup e$.
            \begin{figure}[H]
                \centering
                \includegraphics[height=4cm]{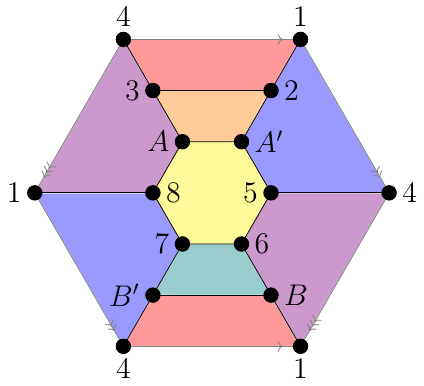} \hspace{.25cm}
                \includegraphics[height=4cm]{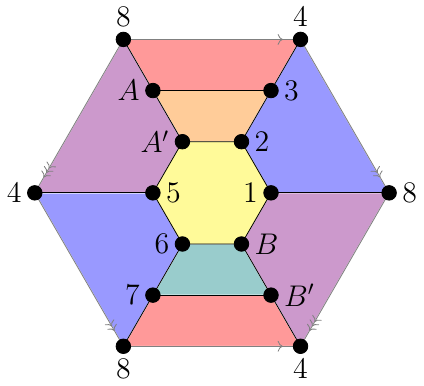} \hspace{.25cm}
                \includegraphics[height=2.5cm]{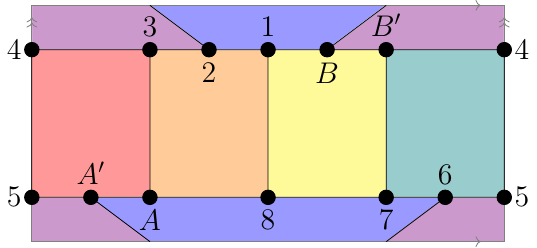}           
                \caption{Three labelled embeddings of $F_{12}$ into the torus}
                \label{fig:F12Labelled}
            \end{figure}

            \vspace{-1cm}
\end{proof}

\begin{prop}
    \label{prop:F13Torus}
    $F_{13} \cup e$ is toroidal for all choices of $e$.
\end{prop}

\begin{proof} We first observe that $F_{13}$ has two types of edges, where the equivalence classes represent the orbits under automorphisms:
        \begin{multicols}{2}
            \begin{itemize}
                \item $[A1]=\{A1, A4, A7, B2, B5, B8, C3, C6, C9\}$
                \item $[12]=\{12, 23, 34, 45, 56, 67, 78, 89, 91\}$
            \end{itemize}
            \newcolumn
            \begin{figure}[H]
                \centering
                \includegraphics[height=3cm]{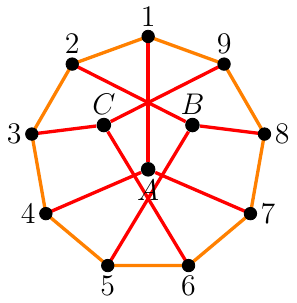}
            \end{figure}
        \end{multicols}

        Similar to Propositions \ref{prop:F11Torus} and \ref{prop:F12Torus}, up to isomorphism there are at most ten ways to add one edge $e$ to $F_{13}$ such that $\girth(F_{13} \cup e) \geq 4$.
        \begin{multicols}{4}
            \begin{enumerate}[(1)]
                \item $A1-B2$
                \item $A1-C3$
                \item $A1-B5$
                \item $A1-23$
                \item $A1-34$
                \item $A1-45$
                \item $A1-56$
                \item $12-34$
                \item $12-45$
                \item $12-56$
            \end{enumerate}
        \end{multicols}
        
        Using Remark \ref{rmk:5New4Cycles}, we observe that the graphs from Cases (1) and (8) are isomorphic.

        By \cite{Kramer24}, there are exactly two inequivalent unlabelled embeddings of $F_{13}$ into the torus. For our purposes, the 9-cycle embeddings suffices. Labelling the vertices on the cycle highlighted in Figure \ref{fig:F13EmbeddingWith9Cycle} starting at any vertex $1$ through $9$ in order, it is easy to check that this embedding extends to an embedding of $F_{13} \cup e$ into the torus in each of the ten cases.
            \begin{figure}[H]
                \centering
                \includegraphics[height=3.5cm]{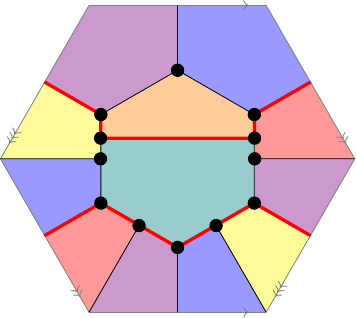}\\
                \caption{Partially labelled embedding of $F_{13}$ into the torus}
                \label{fig:F13EmbeddingWith9Cycle}
            \end{figure}

        \vspace{-1cm}
\end{proof}

\begin{prop}
    \label{prop:F14Torus}
    $F_{14} \cup e$ is toroidal for all choices of $e$.
\end{prop}

\begin{proof}
    We first observe that $F_{14}$ has three types of edges, where the equivalence classes represent the orbits under automorphisms:
        \begin{multicols}{2}
            \begin{itemize}
                \item $[AA']=\{AA', BB'\}$
                \item $[A1]=\{A1, A5, A'3, A'7, B2, B6, B'4, B'8\}$
                \item $[12]=\{12, 23, 34, 45, 56, 67, 78, 81\}$
            \end{itemize}
            \newcolumn
            \begin{figure}[H]
                \centering
                \includegraphics[height=3.5cm]{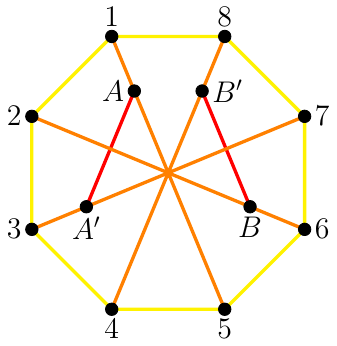}
            \end{figure}
        \end{multicols}        

        Similar to Propositions \ref{prop:F11Torus} - \ref{prop:F13Torus}, up to isomorphism there are at most twelve ways to add one edge $e$ to $F_{14}$ such that $\girth(F_{14} \cup e) \geq 4$.
        \begin{multicols}{4}
            \begin{enumerate}[(1)]
                \item $AA'-BB'$
                \item $AA'-B2$
                \item $AA'-12$
                \item $A1-A'3$
                \item $A1-B2$
                \item $A1-B6$
                \item $A1-23$
                \item $A1-34$
                \item $A1-45$
                \item $12-34$
                \item $12-45$
                \item $12-56$
            \end{enumerate}
        \end{multicols}

        We observe that  Remark \ref{rmk:5New4Cycles} gives the isomorphisms (3) $\cong$ (9) and (5) $\cong$ (10).

    By \cite{Kramer24}, there are exactly two inequivalent unlabelled embeddings of $F_{14}$ into the torus. For our purposes, the 8-cycle embedding suffices. Labelling the vertices on the cycle highlighted in Figure \ref{fig:F14HighlightedEgde} starting at any vertex $1$ through $8$ in order, it is easy to check that this embedding extends to an embedding of $F_{14} \cup e$ into the torus in each of the twelve cases.
            \begin{figure}[H]
                \centering
                \includegraphics[height=3.5cm]{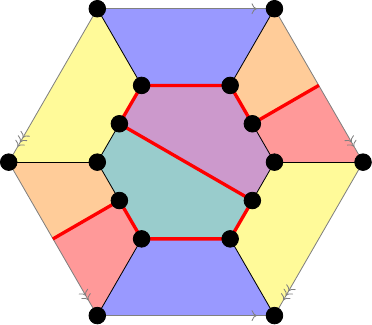}
                \caption{Partially labelled embedding of $F_{14}$ into the torus}
                \label{fig:F14HighlightedEgde}
            \end{figure}

        \vspace{-1cm}
\end{proof}
% ------------------------------------------------------------------------------------------------
% ------------------------------------------------------------------------------------------------

% ------------------------------------------------------------------------------------------------
% ------------------------------------------------------------------------------------------------
\section{Cubic torus obstruction candidates from \texorpdfstring{$G_{1}$}{G\_{1}}}\label{sec:extendingG1}
Recall that $b_1(G_1)=\frac{1}{2}|V(G_1)|+b_0(G_1)=\frac{1}{2}\cdot 10 + 1 = 6$. Hence, as $G_1$ is connected, Remark \ref{rmk:5AddingEdge} implies that $b_1(G_1 \cup e_1) = 7$ and $b_1(G_1 \cup e_1 \cup e_2) = 8$. By \cite{Kramer24}, $G_1$ is toroidal.

The graph $G_1$ has two types of edges, where the equivalence classes represent the orbits under automorphisms:
\begin{multicols}{2}
    \begin{itemize}
        \item $[AA']=\{AA', BB', CC'\}$
        \item $[A1]=\{A1, A2, B1, B2, C1, C2, \\ \text{} \hspace{1.4cm} A'3, A'4, B'3, B'4, C'3, C'4\}$
    \end{itemize}
    We call the edges in $[AA']$ bridges and the edges in $[A1]$ $K_{2,3}-$edges.
    \newcolumn
    \begin{figure}[H]
        \centering
        \includegraphics[height=4cm]{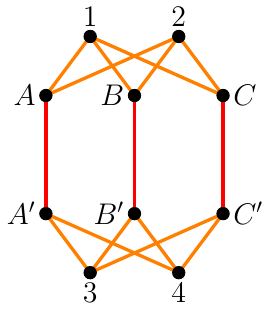}
    \end{figure}
\end{multicols}

In addition to using the symmetries of $G_1$ to detect obvious isomorphisms when attaching one or two edges to $G_1$, we use the following lemma. Observe that $G_1$ is $3$-connected. Recall that a $3$-connected cubic graph is \textit{cyclically $4$-connected} if any cut set of size three consists of edges all incident to a common vertex.

%\newpage
\begin{lemma}
    \label{lem:LessG1Cases}
    The following statements hold.
    \begin{enumerate}[(a)]
        \item If $e$ is an edge connecting a bridge to a $K_{2,3}-$edge, then $G_1 \cup e \cong G_1\cup e'$, where $e'$ connects two $K_{2,3}-$edges in the same $K_{2,3}$.
        \item If $G_1 \cup e_1 \cup e_2$ is cyclically $4$-connected, then $G_1 \cup e_1 \cup e_2 \cong G_1 \cup e'_1 \cup e'_2$ such that $G_1 \cup e'_1$ is cyclically $4$-connected and $e'_2$ attaches to a bridge only if both endpoints attach to a bridge.
        \item If $G_1 \cup e_1 \cup e_2$ is not cyclically $4$-connected, then $G_1 \cup e_1 \cup e_2 \cong G_1 \cup e'_1 \cup e'_2$ such that neither $e'_1$ nor $e'_2$ connect a bridge to a $K_{2,3}-$edge. Moreover, if $e'_2$ is attached to $e'_1$, we may assume its other endpoint attaches to a $K_{2,3}-$edge only if $e'_1$ connects two $K_{2,3}-$edges. Similarly, we may assume its other endpoint attaches to a bridge only if $e'_1$ connects two bridges.
    \end{enumerate}
\end{lemma}

\begin{proof}
    \text{}
    \begin{enumerate}[(a)]
        \item There are two distinct ways to add an edge $e$ to $G_1$ connecting a bridge to a $K_{2,3}-$edge. One endpoint of $e$ attaches to an edge in $[AA']$, without loss of generality assume it attaches to $AA'$. Then the second endpoint either attaches to an edge in $(A1)=\{A1, A2, A'3, A'4\}$ giving $\girth(G_1 \cup e_1)=3$, or to an edge in $(B_1)=\{B1, B2, C1, C2,$ $B'3, B'4, C'3, C'4\}$.           
        If $e=AA'-A1$, then $G_1 \cup e \cong G_1\cup e'$ with $e'=A1-A2$. Note that this also follows from Remark \ref{rmk:CubicYDelta}. Similarly, if $e=AA'-B1$, then $G_1 \cup e \cong G_1 \cup e'$ with $e'=A2-C1$ by Remark \ref{rmk:5New4Cycles}. 
        \begin{figure}[H]
            \centering
            \includegraphics[height=3cm]{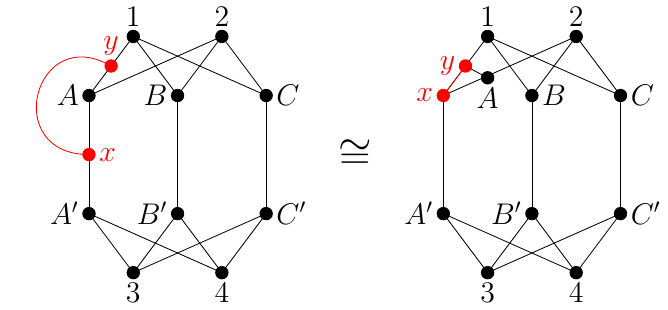} \hspace{.5cm}
            \includegraphics[height=3cm]{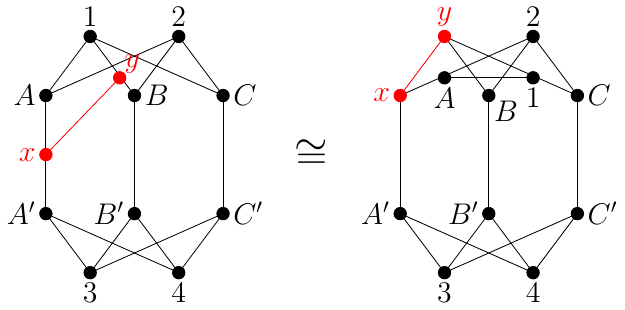}
        \end{figure}

        \item We first observe that it is possible to attach two edges $e_1, e_2$ such that $G_1 \cup e_1 \cup e_2$ is cyclically $4$-connected, but neither $G_1 \cup e_1$ nor $G \cup e_2$ is. For example, $e_1$ connects a bridge to a $K_{2,3}-$edge and $e_2$ connects $e_1$ to an edge in the previously unused $K_{2,3}$ subgraph. It is now easy to see that $G_1 \cup e_1 \cup e_2 \cong G_1 \cup e'_1 \cup e'_2$ such that $G_1 \cup e'_1$ is cyclically $4$-connected and $e'_2$ connects $e'_1$ to a bridge. Another possibility is that $e_1$ and $e_2$ connect the same bridge to a $K_{2,3}-$edge in different $K_{2,3}$ subgraphs.
            Using part (a), we quickly see that $G_1 \cup e_1 \cup e_2 \cong G_1 \cup e'_1 \cup e'_2$ such that $G_1 \cup e'_1$ is cyclically $4$-connected. 
            
            So, we may assume from now on that $G_1 \cup e_1$ is cyclically $4$-connected. In particular, $e_1$ attaches to exactly one edge in each of the $K_{2,3}$ subgraphs. It remains to show that one endpoint of $e_2$ attaches to a bridge only if its other endpoint does as well.
            Suppose that $e_2$ connects a bridge to a $K_{2,3}-$edge. By (a) there exists an isomorphism such that $G_1 \cup e_1 \cup e_2 \cong G_1 \cup e'_1 \cup e'_2$ with $e'_2$ connecting a $K_{2,3}-$edge to a $K_{2,3}-$edge in the same $K_{2,3}$ subgraph, while $e'_1$ either is such that $G \cup e'_1$ is cyclically $4$-connected, or $e'_1$ attaches to $e'_2$. In the latter, it is easy to see that $G_1 \cup e'_1 \cup e'_2 \cong G_1 \cup e''_1 \cup e''_2$ such that $G \cup e''_1$ is cyclically $4$-connected and $e''_2$ connects $e''_1$ to a $K_{2,3}-$edge.
            Finally, suppose that $e_2$ connects a bridge to $e_1$. Using the isomorphisms from (a), we see that $G_1 \cup e_1 \cup e_2 \cong G_1 \cup e'_1 \cup e'_2$, where $G \cup e'_1$ is cyclically $4$-connected and $e'_2$ connects two $K_{2,3}-$edges in the same $K_{2,3}$ subgraph. 

        \item We first show that neither $e'_1$ nor $e'_2$ connect a bridge to a $K_{2,3}-$edge.
            Suppose $e_1$ connects a bridge to a $K_{2,3}-$edge. It follows from (a) that $G_1 \cup e_1 \cup e_2 \cong G_1 \cup e'_1 \cup e'_2$, where $e'_1$ connects two $K_{2,3}-$edges in the same $K_{2,3}$ subgraph. 
            Note that applying an isomorphism from (a) never turns a $K_{2,3}-$edge into a bridge. Moreover, in the not cyclically $4$-connected case, it never turns a new edge $e_i$ into a bridge.
            This observation immediately implies that if neither endpoint of $e_2$ was attached to a bridge, then neither endpoint of $e'_2$ is attached to a bridge.
            If exactly one endpoint of $e_2$ was attached to a bridge, then either neither endpoint of $e'_2$ is attached to a bridge or exactly one endpoint of $e'_2$ is attached to a bridge. In the latter case it follows from (a) that $G_1 \cup e'_1 \cup e'_2 \cong G_1 \cup e''_1 \cup e''_2$, where neither $e''_1$ nor $e''_2$ attaches to a bridge.
            If $e_2$ connected two bridges, then either both endpoints of $e'_2$ are attached to bridges or exactly one endpoint of $e'_2$ is attached to a bridge. In the latter case we can again apply one of the isomorphisms from (a) to obtain $G_1 \cup e'_1 \cup e'_2 \cong G_1 \cup e''_1 \cup e''_2$ such that neither $e''_1$ nor $e''_2$ attach to a bridge.

            In particular, we may assume from now on that $e_1$ either connects two bridges or two $K_{2,3}-$edges. We now consider the cases where $e_2$ attaches to $e_1$.
            First we assume $e_2$ connects $e_1$ to a $K_{2,3}-$edge. If $e_1$ connects two bridges, then $G_1 \cup e_1 \cup e_2 \cong G_1 \cup e'_1 \cup e'_2$, where $e'_1$ connects a bridge to a $K_{2,3}-$edge and $e'_2$ connects $e'_1$ to a bridge. Our previous arguments give an isomorphism $G_1 \cup e'_1 \cup e'_2 \cong G_1 \cup e''_1 \cup e''_2$ where neither $e''_1$ nor $e''_2$ attaches to a bridge. Thus, we may assume $e_1$ connects two $K_{2,3}-$edges.
            Finally, we assume $e_2$ connects $e_1$ to a bridge. If $e_1$ connects two $K_{2,3}-$edges, then $G_1 \cup e_1 \cup e_2 \cong G_1 \cup e'_1 \cup e'_2$, where $e'_1$ connects a bridge to a $K_{2,3}-$edge and $e'_2$ connects $e'_1$ to a $K_{2,3}-$edge. Our previous arguments give an isomorphism $G_1 \cup e'_1 \cup e'_2 \cong G_1 \cup e''_1 \cup e''_2$ where neither $e''_1$ nor $e''_2$ attaches to a bridge. Thus, to obtain a graph that is potentially non-isomorphic to a previously considered graph, we need $e_1$ to connect two bridges.
    \end{enumerate}
    \vspace{-.7cm}
\end{proof}

\begin{prop}
    \label{prop:G1InTorus}
    $G_1 \cup e$ is toroidal for all choices of $e$.
\end{prop}

\begin{proof}
    There are at most six distinct ways to add one edge $e_1$ to $G_1$, two of which give $\girth(G_1 \cup e)=3$. Unlike in the previous sections, we keep track of the cases where adding one edge gives girth three as it is possible that $\girth(G_1 \cup e_1 \cup e_2)=4$ if $\girth(G_1 \cup e_1)=3$. Hence, $G_1 \cup e_1 \cup e_2$ could be an obstruction.

    If one endpoint of $e_1$ attaches to an edge in $[AA']$, without loss of generality assume it attaches to $AA'$, using Lemma \ref{lem:LessG1Cases} we get that its second endpoint attaches to an edge in $(BB')=\{BB', CC'\}$.

    If one endpoint of $e_1$ attaches to an edge in $[A1]$, without loss of generality assume it attaches to $A1$, and assuming the second endpoint does not attach to an edge in $[AA']$, we get the following five classes of edges in $E(G_1)$ for the second endpoint of $e_1$ to attach to.
                \begin{multicols}{3}
                    \begin{itemize}
                        \item $(A2)=\{A2\}$
                        \item $(B1)=\{B1, C1\}$
                        \item $(B2)=\{B2, C2\}$
                        \item $(A'3)=\{A'3, A'4\}$
                        \item $(B'3)=$\\$\{B'3, B'4, C'3, C'4\}$ 
                    \end{itemize}
                \end{multicols}

        Note that the last two graphs are cyclically $4$-connected, while the others are not.
        For each of the six cases, at least one of the two labelled embedding of $G_1$ into the torus pictured below extends to an embedding of $G_1 \cup e$. 
        \begin{figure}[H]
            \centering
            \includegraphics[height=3.5cm]{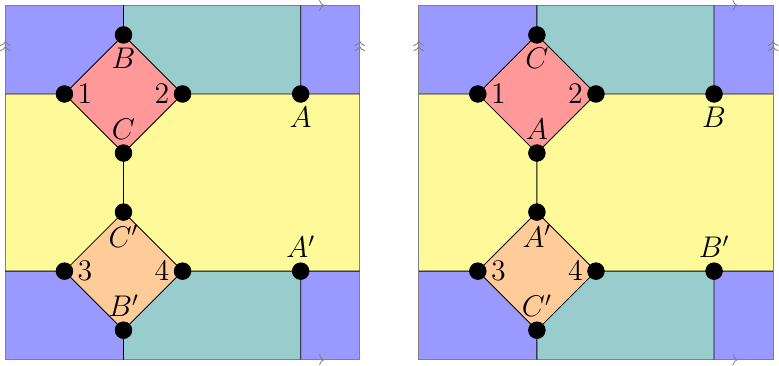}
        \end{figure}
\vspace{-1cm}
\end{proof}

\begin{prop} \label{prop:G1TwoEdgesTorus}
    $G_1 \cup e_1 \cup e_2$ is toroidal if $G_1 \cup e_1 \cup e_2 \notin \{H_4, H_5, \ldots, H_9\}$, where $H_4$ through $H_9$ are pictured below.
    \begin{figure}[H]
        \centering
        \includegraphics[height=2cm]{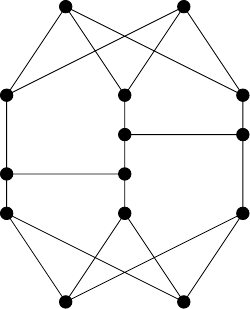} \hspace{.5cm}
            \includegraphics[height=2cm]{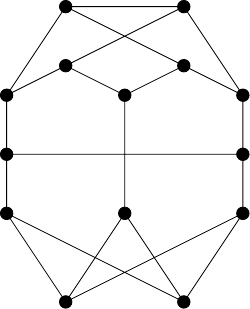} \hspace{.5cm}
            \includegraphics[height=2cm]{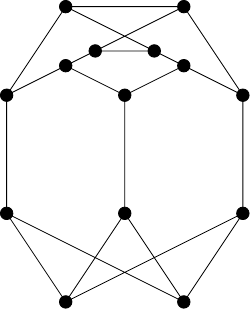} \hspace{.5cm}
            \includegraphics[height=2cm]{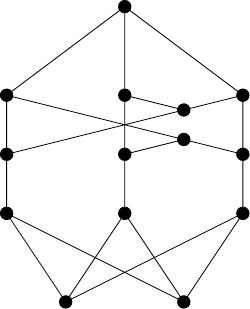} \hspace{.5cm}
            \includegraphics[height=2cm]{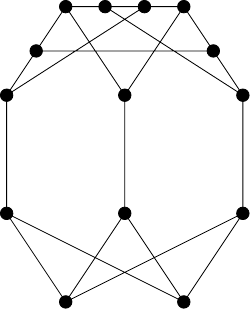} \hspace{.5cm}
            \includegraphics[height=2cm]{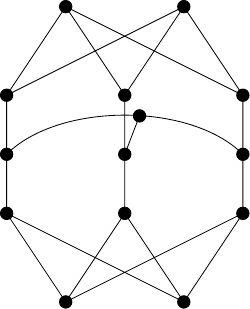}
        \caption{The graphs $H_4$ through $H_9$}
        \label{fig:H4ThroughH9}
    \end{figure}
\end{prop}

Proposition \ref{prop:G1TwoEdgesTorus} will follow immediately from combining Lemmas \ref{lem:4connectedG1CaseA} - \ref{lem:3connectedG1}. In particular, we will show that $G_1 \cup e_1 \cup e_2 \hookrightarrow T^2$ if $G_1 \cup e_1 \cup e_2$ is cyclically $4$-connected.

\begin{lemma}
    \label{lem:4connectedG1CaseA}
    If $G_1\cup e_1$ is cyclically $4$-connected and $e_1$ is an edge in a $4$-cycle in $G_1 \cup e_1$, then $G_1 \cup e_1 \cup e_2$ is toroidal.
\end{lemma}
    
\begin{proof}
    As $e_1$ is part of a $4$-cycle, we may assume without loss of generality that $e_1 = A1-A'3$. $G_1 \cup e_1$ has five types of edges, where the equivalence classes represent the orbits under automorphisms:
            \begin{multicols}{2}
                \begin{itemize}
                    \item $[Ax]=\{Ax, A'y\}$
                    \item $[x1]=\{x1, y3, A2, A'4\}$
                    \item $[B1]=\{B1, B2, C1, C2,  B'3, B'4, C'3, C'4\}$
                    \item $[BB']=\{BB', CC'\}$
                    \item $[xy]=\{xy, AA'\}$
                \end{itemize}
                \begin{figure}[H]
                    \centering
                    \includegraphics[height=3.5cm]{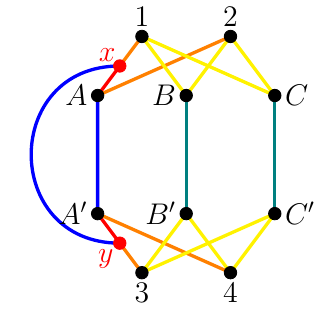}
                \end{figure}
            \end{multicols}

   By Lemma \ref{lem:LessG1Cases} we may assume that $e_2$ attaches to a bridge only if its other endpoint also attaches to a bridge.
            Now, there are at most 18 distinct ways to add one edge $e_2$ to $G_1\cup e_1$ such that $\girth(G_1 \cup e_1 \cup e_2) \geq 4$. We list these below. Note that we are not claiming all of the listed graphs to be pairwise non-isomorphic, but we did eliminate many obvious isomorphisms.
            We postpone cases where $e_2$ attaches to an edge in $[xy]$ until the last case.
            \begin{itemize}
            
                \item If one endpoint of $e_2$ attaches to an edge in $[Ax]$, without loss of generality we always assume it attaches to the given representative, and assuming the second endpoint does not attach to an edge in $[xy]$, we get the following four classes of edges in $E(G_1 \cup e_1)$ for the second endpoint of $e_2$ to attach to. Note that these equivalence classes are the orbits under the subgroup of automorphisms preserving the edge $Ax$. Recall that by Lemma \ref{lem:LessG1Cases}, we may assume that the second endpoint of $e_2$ does not attach to a bridge.
                    \begin{multicols}{2}
                        \begin{enumerate}[(1)]
                            \item $(B1)=\{B1, B2, C1, C2\}$
                            \item $(A'y)=\{A'y\}$
                            %\newcolumn
                            \item $(y3)=\{y3, A'4\}$
                            \item $(B'3)=\{B'3, B'4, C'3, C'4\}$
                        \end{enumerate}
                    \end{multicols}
                    
                \item If one endpoint of $e_2$ attaches to an edge in $[x1]$, and assuming the second endpoint does not attach to an edge in $[Ax]$ or $[xy]$, we get the following five classes of edges in $E(G_1 \cup e_1)$ for the second endpoint of $e_2$ to attach to. Note that we avoid attaching to $y3$ as this gives a graph isomorphic to (2).
                    \begin{multicols}{3}
                        \begin{enumerate}[(1)]
                            \setcounter{enumi}{4}
                            \item $(A2)=\{A2\}$
                            \item $(B2)=\{B2, C2\}$
                            \item $(A'4)=\{A'4\}$
                            \item $(B'3)=\{B'3, C'3\}$
                            \item $(B'4)=\{B'4, C'4\}$
                        \end{enumerate}
                    \end{multicols}
                    
                \item If one endpoint of $e_2$ attaches to an edge in $[B1]$, and assuming the second endpoint does not attach to an edge in $[Ax]$, $[x1]$ or $[xy]$, we get the following five classes of edges in $E(G_1 \cup e_1)$ for the second endpoint of $e_2$ to attach to. 
                    \begin{multicols}{3}
                        \begin{enumerate}[(1)]
                            \setcounter{enumi}{9}
                            \item $(C2)=\{C2\}$
                            \item $(B'3)=\{B'3\}$
                            \item $(B'4)=\{B'4\}$
                            \item $(C'3)=\{C'3\}$
                            \item $(C'4)=\{C'4\}$
                        \end{enumerate}
                    \end{multicols}
                    
                \item If one endpoint of $e_2$ attaches to an edge in $[BB']$, Lemma \ref{lem:LessG1Cases} implies that the second endpoint of $e_2$ attaches to an edge in
                        \begin{enumerate}[(1)]
                            \setcounter{enumi}{14}
                            \item {$(CC')=\{CC'\}$}.
                        \end{enumerate}

                \item Now we assume that one endpoint of $e_2$ attaches to an edge in $[xy]$. We get the following three classes of edges in $E(G_1 \cup e_1)$ for the second endpoint of $e_2$ to attach to. Note that by Lemma \ref{lem:LessG1Cases} we need not consider $e_2$ connecting $e_1$ to a bridge.
                    \begin{multicols}{2}
                        \begin{enumerate}[(1)]
                            \setcounter{enumi}{15}
                            \item {$(A2)=\{A2\}$}
                            \item {$(B1)=\{B1, C1, B'3, C'3\}$}
                            \item {$(B2)=\{B2, C2, B'4, C'4\}$}
                        \end{enumerate}
                    \end{multicols}
            \end{itemize} 

            For each of these graphs, at least one of the following three labelled embeddings of $G_1 \cup e_1$ extends to an embedding of the graph $G_1 \cup e_1 \cup e_2$.
            \begin{figure}[H]
                \centering
                \includegraphics[height=4cm]{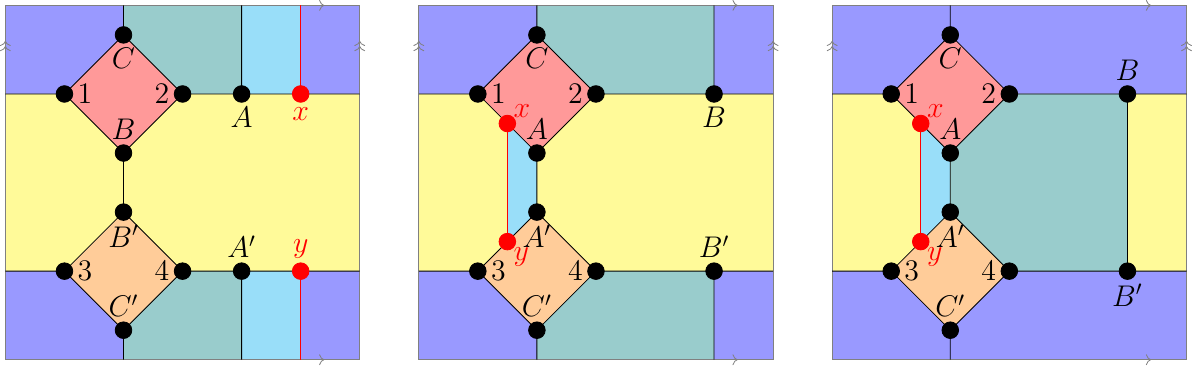}
            \end{figure}
            \vspace{-.5cm}
\end{proof}

\begin{lemma}
    \label{lem:4connectedG1CaseB}
    If $G_1\cup e_1$ is cyclically $4$-connected and $e_1$ is not an edge in a 4-cycle in $G_1 \cup e_1$, then $G_1 \cup e_1 \cup e_2$ is toroidal.
\end{lemma}
    
\begin{proof}
    First note that we may assume without loss of generality that $e_1 = A1-B'3$. To avoid listing graphs already covered in Lemma \ref{lem:4connectedG1CaseA}, we assume that if $G \cup e_2$ is cyclically $4$-connected, $e_2$ is not part of a 4-cycle. Moreover, by Lemma \ref{lem:LessG1Cases} we may assume that $e_2$ attaches to a bridge only if its other endpoint also attaches to a bridge.
    Let $x$ be the newly added vertex on the edge $A1$, and $y$ be the newly added vertex on the edge $B'3$.

    $G_1 \cup e_1$ has ten types of edges, where the equivalence classes represent the orbits under automorphisms:
            \begin{multicols}{3}
                \begin{itemize}
                    \item $[Ax]=\{Ax, B'y\}$
                    \item $[x1]=\{x1, y3\}$
                    \item $[A2]=\{A2, B'4\}$
                    \item $[B1]=\{B1, A'3\}$
                    \item $[B2]=\{B2, A'4\}$
                    \item $[C1]=\{C1, C'3\}$
                    \item $[C2]=\{C2, C'4\}$
                    \item $[AA']=\{AA', BB'\}$
                    \item $[CC']=\{CC'\}$
                    \item $[xy]=\{xy\}$
                \end{itemize}
                \begin{figure}[H]
                    \centering
                    \includegraphics[height=3.75cm]{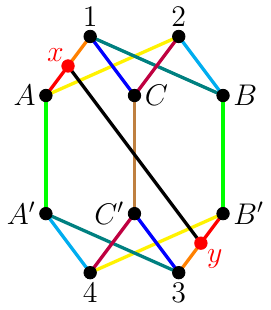}
                \end{figure}
            \end{multicols}

            There are at most $32$ distinct ways to add one edge $e_2$ to $G_1\cup e_1$ such that $\girth(G_1 \cup e_1 \cup e_2) \geq 4$, which we list below. We postpone attaching $e_2$ to $e_1=xy$ until the last case.
            \begin{itemize}
                \item If one endpoint of $e_2$ attaches to an edge in $[Ax]$, without loss of generality we always assume it attaches to the given representative, and assuming the second endpoint does not attach to the edge $xy$, we get the following eight classes of edges in $E(G_1 \cup e_1)$ for the second endpoint of $e_2$ to attach to. Recall that by Lemma \ref{lem:LessG1Cases} we may assume that the second endpoint of $e_2$ does not attach to a bridge. Since we assume $e_2$ not to be part of a new $4$-cycle in $G \cup e_2$ if $G \cup e_2$ is cyclically $4$-connected, we avoid attaching to an edge in $\{A'3, A'4\}$. Moreover, we avoid attaching to $C1$ because using Remark \ref{rmk:5New4Cycles} regarding $4$-cycles this gives a graph isomorphic to $G_1 \cup (A1-A'3) \cup (xy-B1)$, which we covered in Lemma \ref{lem:4connectedG1CaseA}.
                    \begin{multicols}{3}
                        \begin{enumerate}[(1)]
                            \item {$(B1)=\{B1\}$}
                            \item {$(B2)=\{B2\}$}
                            \item {$(C2)=\{C2\}$}
                            \item {$(B'y)=\{B'y\}$}
                            \item {$(y3)=\{y3\}$}
                            \item {$(B'4)=\{B'4\}$}
                            \item {$(C'3)=\{C'3\}$}
                            \item {$(C'4)=\{C'4\}$}
                        \end{enumerate}
                    \end{multicols}
                    
                \item If one endpoint of $e_2$ attaches to an edge in $[x1]$, and assuming the second endpoint of $e_2$ does not attach to an edge in $[Ax]$ or $[xy]$, we get the following six classes of edges in $E(G_1 \cup e_1)$ for the second endpoint of $e_2$ to attach to. We avoid attaching to $y3$ as this gives a graph isomorphic to (4).
                    \begin{multicols}{3}
                        \begin{enumerate}[(1)]
                            \setcounter{enumi}{8}
                            \item {$(A2)=\{A2\}$}
                            \item {$(B2)=\{B2\}$}
                            \item {$(C2)=\{C2\}$}
                            \item {$(B'4)=\{B'4\}$}
                            \item {$(C'3)=\{C'3\}$}
                            \item {$(C'4)=\{C'4\}$}
                        \end{enumerate}
                        \end{multicols}
                    
                \item If one endpoint of $e_2$ attaches to an edge in $[A2]$, and assuming the second endpoint of $e_2$ does not attach to an edge in $[Ax]$, $[x1]$ or $[xy]$, we get the following five classes of edges in $E(G_1 \cup e_1)$ for the second endpoint of $e_2$ to attach to.
                    \begin{multicols}{3}
                        \begin{enumerate}[(1)]
                            \setcounter{enumi}{14}
                            \item {$(B1)=\{B1\}$}
                            \item {$(C1)=\{C1\}$}
                            \item {$(B'4)=\{B'4\}$}
                            \item {$(C'3)=\{C'3\}$}
                            \item {$(C'4)=\{C'4\}$}
                        \end{enumerate}
                    \end{multicols}
                    
                \item If one endpoint of $e_2$ attaches to an edge in $[B1]$, and assuming the second endpoint of $e_2$ does not attach to an edge in $[Ax]$, $[x1]$, $[A2]$ or $[xy]$, we get the following five classes of edges in $E(G_1 \cup e_1)$ for the second endpoint of $e_2$ to attach to.
                    \begin{multicols}{3}
                        \begin{enumerate}[(1)]
                            \setcounter{enumi}{19}
                            \item {$(C2)=\{C2\}$}
                            \item {$(A'3)=\{A'3\}$}
                            \item {$(A'4)=\{A'4\}$}
                            \item {$(C'3)=\{C'3\}$}
                            \item {$(C'4)=\{C'4\}$}
                        \end{enumerate}
                    \end{multicols}
                    
                \item If one endpoint of $e_2$ attaches to an edge in $[B2]$, and assuming the second endpoint of $e_2$ does not attach to an edge in $[Ax]$, $[x1]$, $[A2]$, $[B1]$ or $[xy]$, we get the following four classes of edges in $E(G_1 \cup e_1)$ for the second endpoint of $e_2$ to attach to.
                    \begin{multicols}{3}
                        \begin{enumerate}[(1)]
                            \setcounter{enumi}{24}
                            \item {$(C1)=\{C1\}$}
                            \item {$(A'4)=\{A'4\}$}
                            \item {$(C'3)=\{C'3\}$}
                            \item {$(C'4)=\{C'4\}$}
                        \end{enumerate}
                    \end{multicols}
                    
                \item If one endpoint of $e_2$ attaches to an edge in $[C1]$ or $[C2]$, and assuming the second endpoint of $e_2$ does not attach to an edge in $[Ax]$, $[x1]$, $[A2]$, $[B1]$, $[B2]$ or $[xy]$, we get no classes of edges in $E(G_1 \cup e_1)$ for the second endpoint of $e_2$ to attach to.
                    
                \item If one endpoint of $e_2$ attaches to an edge in $[AA']$, recall from Lemma \ref{lem:LessG1Cases} that we assume its second endpoint attaches to a bridge as well. Thus, we get the following two classes of edges in $E(G_1 \cup e_1)$ for the second endpoint of $e_2$ to attach to.
                    \begin{multicols}{3}
                        \begin{enumerate}[(1)]
                            \setcounter{enumi}{28}
                            \item {$(BB')=\{BB'\}$}
                            \item {$(CC')=\{CC'\}$}
                        \end{enumerate}
                    \end{multicols}
                    
                \item If one endpoint of $e_2$ attaches to $CC'$, we again may assume its second endpoint attaches to a bridge as well. Assuming it does not attach to an edge in $[AA']$, we get no classes of edges in $E(G_1 \cup e_1)$ for the second endpoint of $e_2$ to attach to.
                    
                \item If one endpoint of $e_2$ attaches to $xy$, recall that by Lemma \ref{lem:LessG1Cases}, we may assume its second endpoint does not attach to a bridge. Additionally, we avoid attaching to an edge in $[B1]$ or $[B2]$ as this would give a graph isomorphic to one of the graphs covered in Lemma \ref{lem:4connectedG1CaseA}. Moreover, we avoid attaching to $C1$ as this gives a graph isomorphic to graph (1), where we have used Remark \ref{rmk:5New4Cycles} to obtain the isomorphism. Thus, we get the following two classes of edges in $E(G_1 \cup e_1)$ for the second endpoint of $e_2$ to attach to. 
                    \begin{multicols}{3}
                        \begin{enumerate}[(1)]
                            \setcounter{enumi}{30}
                            \item {$(A2)=\{A2, B'4\}$}
                            \item {$(C2)=\{C2, C'4\}$}
                        \end{enumerate}
                    \end{multicols}
                    
            \end{itemize}
            
        For each of these graphs, at least one of the following five labelled embeddings of $G_1 \cup e_1$ extends to an embedding of $G_1 \cup e_1 \cup e_2$.
            \begin{figure}[H]
                \centering
                \includegraphics[width=\textwidth]{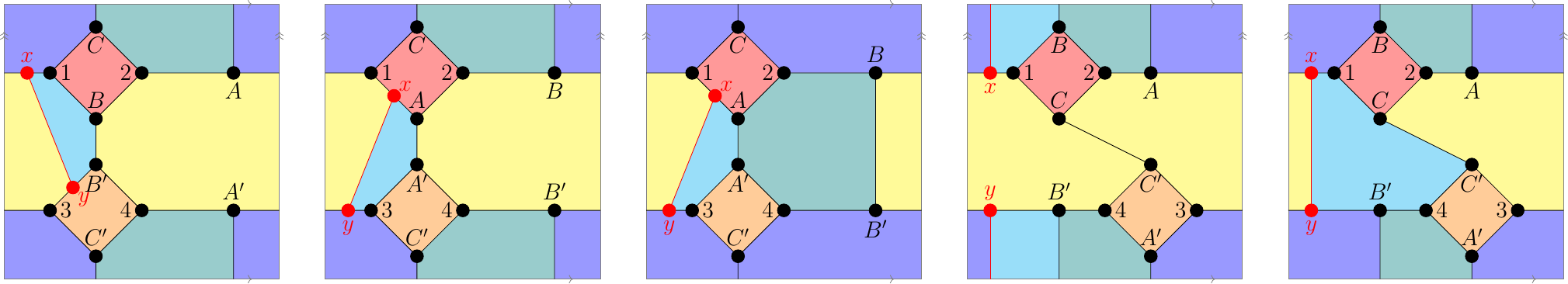}
            \end{figure}
            \vspace{-.75cm}
\end{proof}

\begin{lemma}
    \label{lem:3connectedG1}
    If $G_1 \cup e_1$ is not cyclically $4$-connected and $G_1 \cup e_1 \cup e_2 \notin \{H_4, H_5, \ldots, H_9\}$, then $G_1 \cup e_1 \cup e_2$ is toroidal.
\end{lemma}

\begin{proof}
    As $G_1 \cup e_1$ is not cyclically $4$-connected, it follows from Lemma \ref{lem:LessG1Cases} that we may assume $G_1 \cup e_1 \cup e_2$ to not be cyclically $4$-connected either. Our previous work gives that $e_1 \in \{A1-A2, A1-B1, A1-B2, AA'-BB'\}$. As $G_1 \cup e_1 \cup e_2$ is not cyclically $4$-connected, $e_2$ does not connect the two different $K_{2,3}$ subgraphs in $G_1$. By Lemma \ref{lem:LessG1Cases}, we may assume that $e_2$ does not connect a bridge to a $K_{2,3}-$edge. Moreover, we may assume that $e_2$ connects $e_1$ to a bridge if and only if $e_1=AA'-BB'$. As before, we are not claiming that the listed graphs are pairwise non-isomorphic, but we did eliminate many obvious isomorphisms.

    \textit{Case A:} $e_1 = A1-A2$\\
        Let $x$ be the vertex subdividing the edge $A1$, and $y$ be the vertex subdividing the edge $A2$.
        Note that $\girth(G_1 \cup e_1) = 3$, where $A-x-y-A$ forms a 3-cycle. Since we are only interested in graphs with girth at least four, the edge $e_2$ needs to break up the 3-cycle.
        Under this additional assumption, $G_1 \cup e_1$ has two types of edges for $e_2$ to attach to, where the equivalence classes represent orbits under automorphisms:
        \begin{multicols}{2}
            \begin{itemize}
                \item $[Ax]=\{Ax, Ay\}$
                \item $[xy] = \{xy\}$
            \end{itemize}
            \newcolumn
            \begin{figure}[H]
                \centering
                \includegraphics[height=3cm]{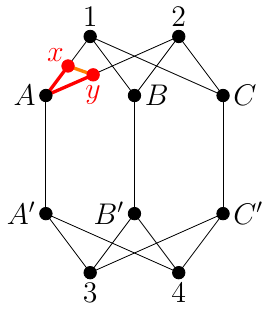}
            \end{figure}
        \end{multicols}

        Using Lemma \ref{lem:LessG1Cases} and the fact that $G_1 \cup e_1 \cup e_2$ is not cyclically $4$-connected, we quickly see that up to isomorphism there are at most three possible placements of the edge $e_2$ such that $\girth(G_1 \cup e_1 \cup e_2) \geq 4$.
        \begin{multicols}{3}
            \begin{enumerate}[(1)]
                \item $Ax-y2$
                \item $Ax-B1$
                \item $Ax-B2$
            \end{enumerate}
        \end{multicols}
        Note that if one endpoint of $e_2$ were to attach to $xy$, the second endpoint would need to attach to an edge in $(B1)=\{B1, B2, C1, C2\}$. Using Remark \ref{rmk:5New4Cycles} we obtain that this graph is isomorphic to (2).
        
        For all these graphs, the labelled embedding of $G_1 \cup e_1$ pictured below extends to an embedding of $G_1 \cup e_1 \cup e_2$.
            \begin{figure}[H]
                \centering
                \includegraphics[height=3.5cm]{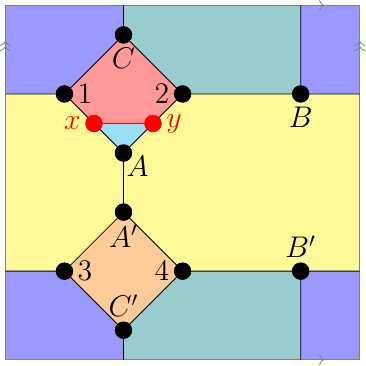} 
            \end{figure}
        
    \textit{Case B:} $e_1 = A1-B1$\\
        Let $x$ be the vertex subdividing the edge $A1$, and $y$ be the vertex subdividing the edge $B1$.
        Similarly to Case A, the edge $e_2$ needs to break up the 3-cycle $1-x-y-1$.
        Under this additional assumption, $G_1 \cup e_1$ has two types of edges for $e_2$ to attach to, where the equivalence classes represent orbits under automorphisms:
        \begin{multicols}{2}
            \begin{itemize}
                \item $[x1]=\{x1, y1\}$
                \item $[xy] = \{xy\}$
                \newcolumn
            \end{itemize}
            \begin{figure}[H]
                \centering
                \includegraphics[height=3cm]{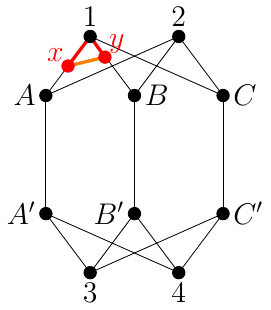}
            \end{figure}
        \end{multicols}

        To avoid listing graphs already covered in Case A, we will assume that $G_1 \cup e_1 \cup e_2 \not\cong G_1 \cup e'_1 \cup e'_2$, where $G_1 \cup e'_1$ contains a three cycle involving one of the vertices $A, B,$ or $C$. We now list the five graphs we get from attaching the second edge $e_2$ such that $\girth(G_1 \cup e_1 \cup e_2) \geq 4$. 
        \begin{itemize}
            \item If one endpoint of $e_2$ attaches to an edge in $[x1]$, without loss of generality assume it attaches to $x1$, we get the following three classes of edges in $E(G_1 \cup e_1)$ for the second endpoint of $e_2$ to attach to. 
            We do not attach to $A2$, as this would result in a graph isomorphic to (2) from Case A. Moreover, by Lemma \ref{lem:LessG1Cases}, we avoid attaching to a bridge.
                \begin{multicols}{3}
                    \begin{enumerate}[(1)]
                        \setcounter{enumi}{3}
                        \item $(By)=\{By\}$
                        \item $(B2)=\{B2\}$
                        \item $(C2)=\{C2\}$
                    \end{enumerate}
                \end{multicols}
                
            \item If one endpoint of $e_2$ attaches to $xy$, and assuming the second endpoint does not attach to an edge in $[A2]=\{A2, B2\}$ as this would result in a graph isomorphic to one from Case A, we get the following two classes of edges in $E(G_1 \cup e_1)$ for the second endpoint of $e_2$ to attach to.
                \begin{multicols}{3}
                    \begin{enumerate}[(1)]
                        \setcounter{enumi}{6}
                            \item $(C1)=\{C1\}$
                            \item $(C2)=\{C2\}$
                        \end{enumerate}
                    \end{multicols}
        \end{itemize}

        For graph (6), the labelled embedding of $G_1 \cup e_1$ pictured below extends to an embedding of $G_1 \cup e_1 \cup e_2$. Recall the graphs $H_7$ and $H_9$ from Figure \ref{fig:H4ThroughH9}. We observe that (4) $\cong H_9 \cong$ (7) and (5) $\cong H_7 \cong$ (8).
            \begin{figure}[H]
                \centering
                \includegraphics[height=3.5cm]{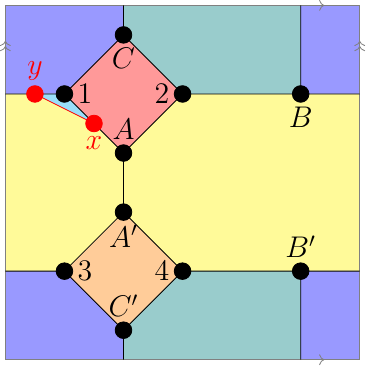}
                \end{figure}
                
    \textit{Case C:} $e_1 = A1-B2$\\
        Let $x$ be the vertex subdividing the edge $A1$, and $y$ be the vertex subdividing the edge $B2$.
        The graph $G_1 \cup e_1$ has nine types of edges, where the equivalence classes represent orbits under automorphisms:
        \begin{multicols}{3}
            \begin{itemize}
                \item $[Ax]=\{Ax, By\}$
                \item $[x1]=\{x1, y2\}$
                \item $[A2]=\{A2, B1\}$
                \item $[C1]=\{C1, C2\}$
                \item $[A'3]=\{A'3, A'4, B'3, B'4\}$
                \item $[C'3]=\{C'3, C'4\}$
                \item $[AA']=\{AA', BB'\}$
                \item $[CC']=\{CC'\}$
                \item $[xy] = \{xy\}$
            \end{itemize}
            \newcolumn
            \begin{figure}[H]
                \centering
                \includegraphics[height=3cm]{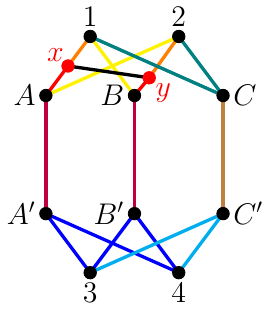}
            \end{figure}
        \end{multicols}

        To avoid listing graphs already covered in Case A - B, we will assume that $\girth(G \cup e_2) \geq 4$.
        We now list the ten graphs we get from attaching the second edge $e_2$ such that $\girth(G_1 \cup e_1 \cup e_2) \geq 4$. As before, we only allow attaching $e_2$ to $e_1 = xy$ in the last case.
        \begin{itemize}
        
            \item If one endpoint of $e_2$ attaches to an edge in $[Ax]$, without loss of generality we always assume it attaches to the given representative, and assuming the second endpoint does not attach to the edge $xy$, we get the following three classes of edges in $E(G_1 \cup e_1)$ for the second endpoint of $e_2$ to attach to. 
                \begin{multicols}{3}
                    \begin{enumerate}[(1)]
                        \setcounter{enumi}{8}
                        \item $(By)=\{By\}$
                        \item $(y2)=\{y2\}$
                        \item $(C2)=\{C2\}$
                    \end{enumerate}
                \end{multicols}
                
            \item If one endpoint of $e_2$ attaches to an edge in $[x1]$, and assuming the second endpoint of $e_2$ does not attach to an edge in $[Ax]$ or $[xy]$, the second endpoint of $e_2$ attaches to an edge in
                    \begin{enumerate}[(1)]
                        \setcounter{enumi}{11}
                        \item $(C2)=\{C2\}$.
                    \end{enumerate}
                
            \item If one endpoint of $e_2$ attaches to an edge in $[A2]$, and assuming the second endpoint of $e_2$ does not attach to an edge in $[Ax], [x1]$ or $[xy]$, we get the following two classes of edges in $E(G_1 \cup e_1)$ for the second endpoint of $e_2$ to attach to.
                \begin{multicols}{3}
                    \begin{enumerate}[(1)]
                        \setcounter{enumi}{12}
                        \item $(B1)=\{B1\}$
                        \item $(C1)=\{C1\}$
                    \end{enumerate}
                \end{multicols}
                
            \item If one endpoint of $e_2$ attaches to an edge in $[C1]$, and assuming the second endpoint of $e_2$ does not attach to an edge in $[Ax], [x1], [A2]$ or $[xy]$, we get the no classes of edges in $E(G_1 \cup e_1)$ for the second endpoint of $e_2$ to attach to.
                        
            \item If one endpoint of $e_2$ attaches to an edge in $[A'3]$, and assuming the other endpoint does not attach to an edge in $[xy]$, we get the following two classes of edges in $E(G_1 \cup e_1)$ for the second endpoint of $e_2$ to attach to.
                \begin{multicols}{3}
                    \begin{enumerate}[(1)]
                        \setcounter{enumi}{14}
                        \item $(B'4)=\{B'4\}$
                        \item $(C'4)=\{C'4\}$
                    \end{enumerate}
                \end{multicols}

            \item If one endpoint of $e_2$ attaches to an edge in $[C'3]$, and assuming the second endpoint of $e_2$ does not attach to an edge in $[A'3]$ or $[xy]$. we get the no classes of edges in $E(G_1 \cup e_1)$ for the second endpoint of $e_2$ to attach to.
                
            \item If one endpoint of $e_2$ attaches to an edge in $[AA']$, then by Lemma \ref{lem:LessG1Cases}, the second endpoint attaches to a bridge as well. We get the following two classes of edges in $E(G_1 \cup e_1)$ for the second endpoint of $e_2$ to attach to.
                \begin{multicols}{3}
                    \begin{enumerate}[(1)]
                        \setcounter{enumi}{16}
                        \item $(BB')=\{BB'\}$
                        \item $(CC')=\{CC'\}$
                    \end{enumerate}
                \end{multicols}
                
            \item If one endpoint of $e_2$ attaches to $CC'$, and assuming the second endpoint of $e_2$ does not attach to an edge in $[AA']$ or $[xy]$, we get the no classes of edges in $E(G_1 \cup e_1)$ for the second endpoint of $e_2$ to attach to.
                
            \item If one endpoint of $e_2$ attaches to $xy$, we avoid attaching to an edge in $[A2]$ as $G_1 \cup (A1-B2) \cup (xy-A2) \cong G_1 \cup (A1-A2) \cup (xy-B1)$, which is a graph discussed in Case A. Similarly, we avoid attaching to an edge in $[C1]$ as these graphs were discussed in Case B. We get no classes of edges in $E(G_1 \cup e_1)$ for the second endpoint of $e_2$ to attach to.
        
        \end{itemize}

        For graphs (10) - (12), (14) - (16), and (18) at least one of the following two labelled embeddings of $G_1 \cup e_1$ extends to an embedding of $G_1 \cup e_1 \cup e_2$. Recall the graphs $H_5, H_6,$ and $H_8$ from Figure \ref{fig:H4ThroughH9}. We observe that (9) $\cong H_6$, (13) $\cong H_8$, and (17) $\cong H_5$.
        \begin{figure}[H]
            \centering
            \includegraphics[height=3.5cm]{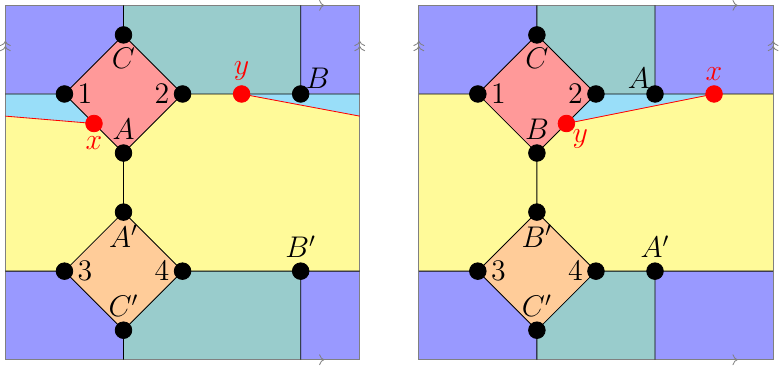}
        \end{figure}

    \textit{Case D:} $e_1 = AA'-BB'$\\
        Let $x$ be the vertex subdividing the edge $AA'$, and $y$ be the vertex subdividing the edge $BB'$.\\
        To avoid listing graphs already covered in Cases A - C, we will assume that $e_2$ does not attach to a $K_{2,3}-$edge.
        Under this additional assumption, $G_1 \cup e_1$ has three types of edges for $e_2$ to attach to, where the equivalence classes represent orbits under automorphisms:
        \begin{multicols}{2}
            \begin{itemize}
                \item $[Ax]=\{Ax, xA', By, yB'\}$
                \item $[CC']=\{CC'\}$
                \item $[xy]=\{xy\}$
            \end{itemize}
            \newcolumn
            \begin{figure}[H]
                \centering
                \includegraphics[height=3cm]{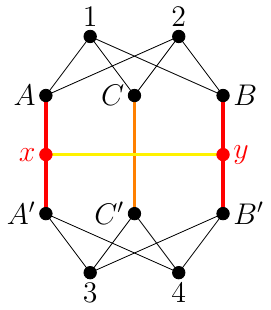}
            \end{figure}
        \end{multicols}

        We quickly see that we obtain at most four pairwise non-isomorphic graphs by attaching the second edge $e_2$ such that $\girth(G_1 \cup e_1 \cup e_2) \geq 4$.
        \begin{multicols}{4}
            \begin{enumerate}[(1)]
            \setcounter{enumi}{18}
                \item $Ax-By$
                \item $Ax-yB'$
                \item $Ax-CC'$
                \item $CC'-xy$
            \end{enumerate}
        \end{multicols}
       
    For graphs (19) and (20), the labelled embedding of $G_1 \cup e_1$ pictured below extends to an embedding of $G_1 \cup e_1 \cup e_2$. Recall the graphs $H_4$ and $H_9$ from Figure \ref{fig:H4ThroughH9}.
        We observe that (21) $\cong H_4$ and (22) $\cong H_9$.
            \begin{figure}[H]
                \centering
                \includegraphics[height=3.5cm]{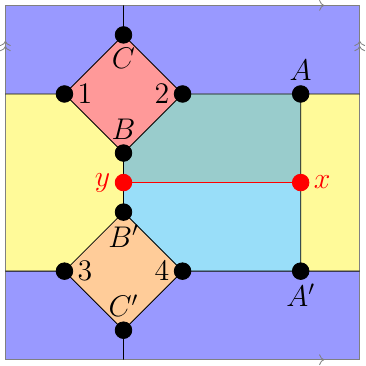}
            \end{figure}
        \vspace{-.5cm}
\end{proof}
% ------------------------------------------------------------------------------------------------
% ------------------------------------------------------------------------------------------------

% ------------------------------------------------------------------------------------------------
% ------------------------------------------------------------------------------------------------
\section{Properties of \texorpdfstring{$H_0, H_1, \ldots, H_9$}{H\_0, H\_1, ..., H\_9}}\label{sec:defHi}
In this section, we define the graphs $H_0, H_1, \ldots, H_9$, and $H_0 \cup e$. We show that they are pairwise non-isomorphic and that any proper subgraph of a graph in $\{H_i\}_{i=0}^9$ does embed into the torus. Moreover, we prove none of $H_0$ through $H_9$ embeds into the torus.
This shows that $H_0, H_1, \ldots, H_9$ are cubic torus obstructions with Betti number at most eight.
Together with Proposition \ref{prop:ProjectivePlanarToroidal} and Sections \ref{sec:t2obstruction_candidates} - \ref{sec:extendingG1} this gives that $\{H_i\}_{i=0}^9$ is the complete set of cubic torus obstructions of Betti number at most eight. Additionally, $H_0 \cup e$ is the unique graph of Betti number at most eight not embeddable into the torus without being an obstruction. Thus, we will have proven the main theorem.

Recall from Sections \ref{sec:t2obstruction_candidates} - \ref{sec:extendingG1} that there are at most eleven pairwise non-isomorphic cubic graphs $H_0 \cup e, H_0, H_1, \ldots, H_9$ with Betti number at most eight that do not embed into the torus. We have pictured the graphs $H_0, H_1, \ldots, H_9$ below.

\begin{figure}[H]
    \centering
    \includegraphics[height=1.9cm]{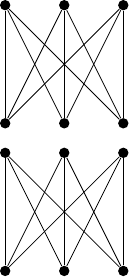}
    \includegraphics[height=1.9cm]{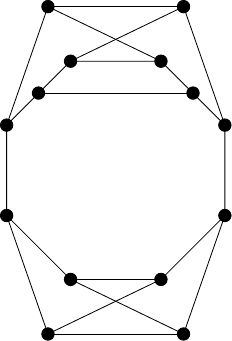}
    \includegraphics[height=1.9cm]{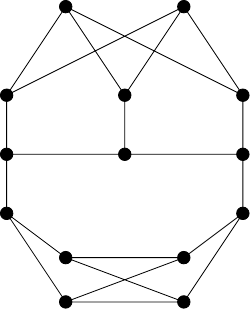}
    \includegraphics[height=1.9cm]{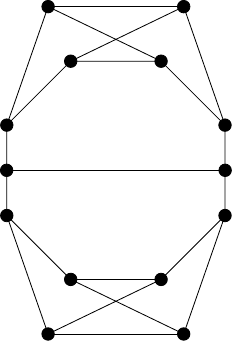}
    \includegraphics[height=1.9cm]{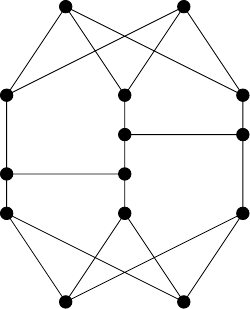}
    \includegraphics[height=1.9cm]{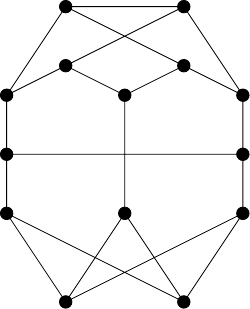}
    \includegraphics[height=1.9cm]{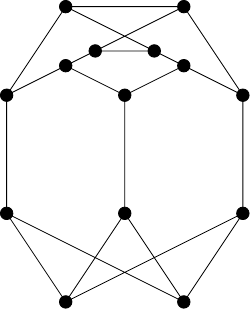}
    \includegraphics[height=1.9cm]{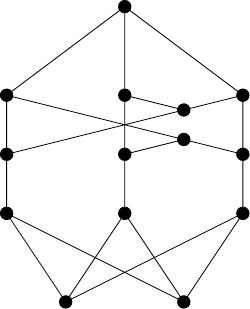}
    \includegraphics[height=1.9cm]{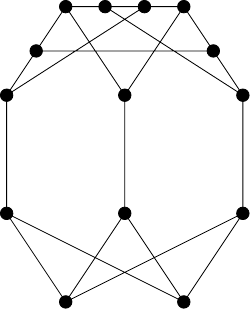}
    \includegraphics[height=1.9cm]{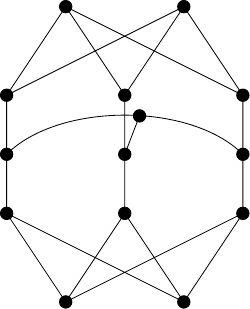}
    %\caption{The graphs $H_0$ through $H_9$}
\end{figure}

Moreover, recall that $H_0 \cong E_{42}$ as described in Section \ref{sec:t2obstruction_candidates}, $H_1, H_2, H_3, H_4$ are of the form $F_{11} \cup e$ as described in Section \ref{sec:extendingF11-F14}, and $H_4, H_5, \ldots, H_9$ are of the form $G_1 \cup e_1 \cup e_2$ as described in Section \ref{sec:extendingG1}. Note in particular that $H_4$ contains subgraphs homeomorphic to $F_{11}$ and $G_1$.

\begin{prop}
    $H_0 \cup e, H_0, H_1, \ldots, H_9$ are pairwise non-isomorphic.
\end{prop}
\begin{proof}
    We observe that among the ten graphs, $H_0$ is the only disconnected one and $H_0 \cup e$ is the only graph of connectivity one. Additionally, the number of $4$-cycles in $H_1, H_2, \ldots, H_9$ is seven, eight, ten, six, five, four, four, three, and six, respectively. Hence, it remains to distinguish $H_4$ from $H_9$ and $H_6$ from $H_7$. For these four graphs, we count their $5$-cycles and find exactly four and zero, respectively, for the first pair of graphs and five and four, respectively, for the second pair.    
\end{proof}

\begin{prop}
    Any proper subgraph of $H_i$ for $i \in \{0, 1, \ldots, 9\}$ embeds into the torus.
\end{prop}
\begin{proof}
    Let $i \in \{0,1,\ldots,9\}$. A quick calculation shows that $b_1(H_i)=8$. Moreover, recall that each component of $H_i$ is $2$-connected. Let $e \in H_i$, and let $G$ be a cubic graph such that $G \simeq H_i \setminus e$. Then $G$ has the same number of components as $H_i$ and so $b_1(G)=7$. If $G$ is projective planar, Proposition \ref{prop:ProjectivePlanarToroidal} implies that $G$ is toroidal. If $G$ is not projective planar, \cite{Kramer24} combined with Proposition \ref{prop:G1InTorus} gives that $G$ is toroidal.
\end{proof}

We have already seen in Section \ref{sec:t2obstruction_candidates} that neither $H_0$ nor $H_0 \cup e$ embed into the torus. In particular, $H_0 \cup e$ cannot be an obstruction for the torus as its subgraph $H_0$ does not embed into the torus either.

Next, we show that $H_1, H_2, H_3$, and $H_4$ do not embed into the torus.
Recall that each of these graphs contains a subgraph homeomorphic to $F_{11}$.
Moreover, by \cite{Kramer24} there are exactly two inequivalent embeddings of $F_{11}$ into the torus.
\begin{figure}[H]
    \centering
    \includegraphics[height=3.5cm]{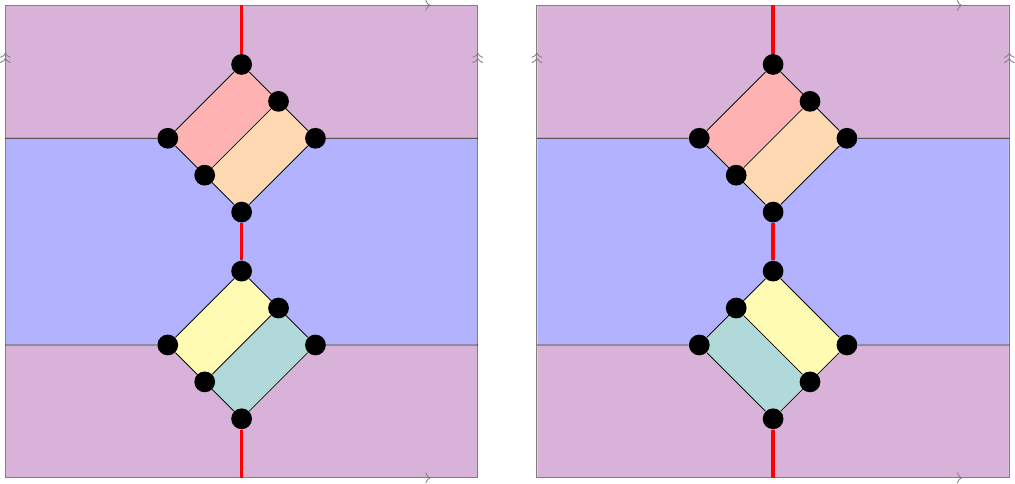}
    \caption{The two unlabelled embeddings $\varphi$ and $\psi$ of $F_{11}$ into the torus}
    \label{fig:F11InTorusUnlabelled}
\end{figure}
We now examine the choices for labelled embeddings of $F_{11}$ into the torus.
Recall two embeddings $\varphi_1$ and $\varphi_2$ of a graph $G$ into a surface $\Sigma$ are \emph{equivalent} if there exists a homeomorphism of $\Sigma$ mapping the image of $G$ under $\varphi_2$ to the image of $G$ under $\varphi_1$.
Note that this definition extends to labelled embeddings in the natural way.
As a first step, we show that up to equivalence we have control over the labelling of some of vertices in the embeddings of $F_{11}$.

    \begin{lemma}
    \label{lem:Fixing2inF11}
    Given an embedding $\varphi: F_{11} \hookrightarrow T^2$, there exists a homeomorphism of the torus $h: T^2 \rightarrow T^2$ such that $\widetilde\varphi:=h \circ \varphi$ is as in Figure \ref{fig:F11LabelledOptions}, where the orange vertices are $\{\widetilde\varphi(A), \widetilde\varphi(1)\}$ and the teal ones are $\{\widetilde\varphi(A'), \widetilde\varphi(8)\}$.
    \begin{figure}[H]
        \centering
        \includegraphics[height=4.5cm]{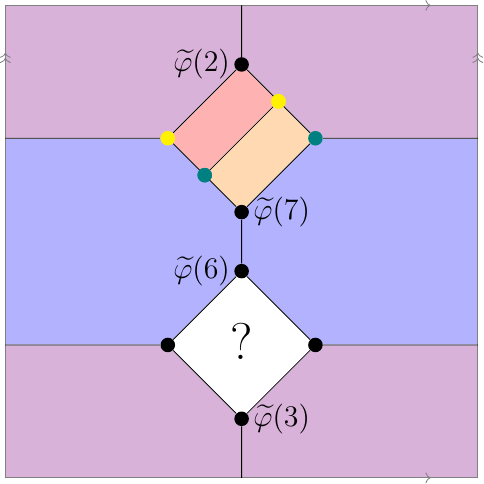}
        \caption{Partially labelled embedding of $F_{11}$}
        \label{fig:F11LabelledOptions}
    \end{figure}
\end{lemma}

\begin{proof}
    Consider the labelling of $F_{11}$ as depicted in Figure \ref{fig:F11Labelled}. First note that the pair of edges $\{23, 67\}$ is the unique minimum cut set of $F_{11}$.
        Thus, one of the edges highlighted red in Figure \ref{fig:F11InTorusUnlabelled} is the edge $23$, while the other one is the edge $67$.
        \begin{figure}[H]
            \centering
            \includegraphics[height=3.5cm]{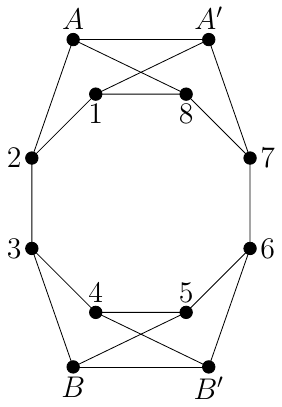}
            \caption{Labelled $F_{11}$}
            \label{fig:F11Labelled}
        \end{figure}

        We will focus on the Vertex $2$. Once we have chosen in which of the four possible locations illustrated below
        %in Figure \ref{fig:F11PossiblePlaces2}
        $\varphi(2)$ lands, the placements of $\varphi(3), \varphi(6),$ and  $\varphi(7)$ are completely determined.
        \begin{figure}[H]
            \centering
            \includegraphics[width=\textwidth]{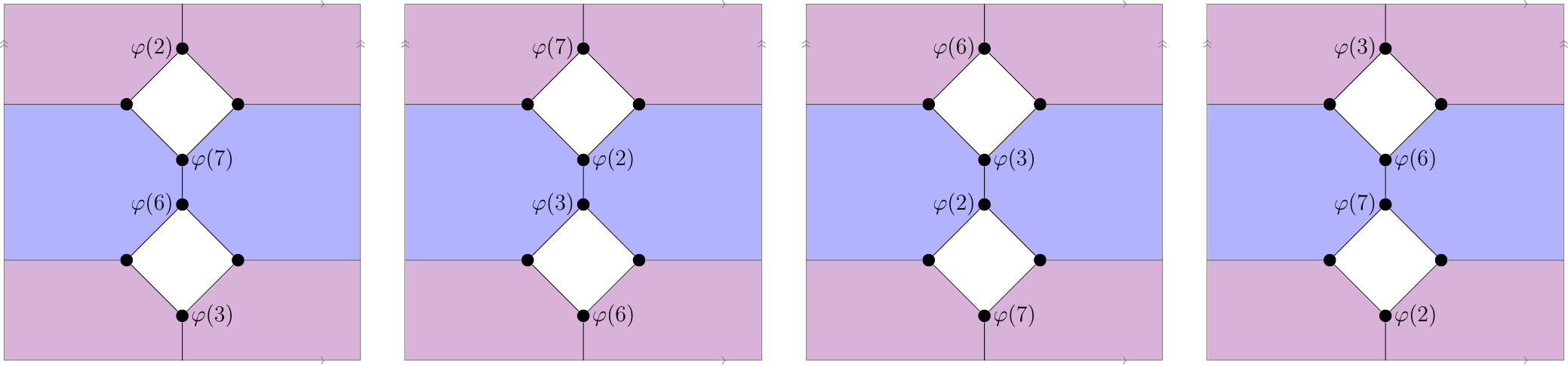}
            %\caption{Options for placing $\varphi(2)$}
            %\label{fig:F11PossiblePlaces2}
        \end{figure}
        If $\varphi(2)$ is not in the spot indicated in Figure \ref{fig:F11LabelledOptions}, we can use a vertical reflection and/or cut our picture of torus along a horizontal line going through the middle and identify the current top and bottom curves to move $\varphi(2)$ to the correct position. Note that the latter operation can be realised by a rotation of the torus when thought of as a 3-dimensional donut.
        As these operations might change the orientation of the adjacent facial $4$-cycles at the top of Figure \ref{fig:F11LabelledOptions}, we can use a horizontal reflection if necessary to ensure that $\widetilde\varphi$ has the same image as $\varphi$ except possibly for the ambiguity displayed in Figure \ref{fig:F11LabelledOptions}.
        Moreover, as Vertex $2$ is adjacent to Vertices $1, 3,$ and $A$, the orange vertices in Figure \ref{fig:F11LabelledOptions} are the images of Vertices $1$ and $A$. Similarly, as Vertex $7$ is adjacent to Vertices $6, 8,$ and $A'$, the turquoise vertices in Figure \ref{fig:F11LabelledOptions} are the images of Vertices $8$ and $A'$.
\end{proof}

While Lemma \ref{lem:Fixing2inF11} does not completely determine the placement of all vertices, and not even which unlabelled embedding of $F_{11}$ we are working with, the information gathered suffices to prove Proposition \ref{prop:H1-H4NotInTorus}.

\begin{prop}
    \label{prop:H1-H4NotInTorus}
    $H_i$ is not toroidal for $i\in\{1, \ldots, 4\}$.
\end{prop}

\begin{proof}
    By Lemma \ref{lem:Fixing2inF11}, we may assume the images of Vertices $2, 3, 6,$ and $7$ to be in the locations indicated in Figure \ref{fig:F11LabelledOptions}.
     Recall that $H_1 \cong F_{11} \cup (\textcolor{red}{A2}-\textcolor{blue}{A'7})$, $H_2 \cong F_{11} \cup (\textcolor{red}{A2}-\textcolor{green}{67})$, $H_3 \cong F_{11} \cup (\textcolor{brown}{23}-\textcolor{green}{67})$, and  $H_4 \cong F_{11} \cup (\textcolor{red}{A2}-\textcolor{violet}{B'6})$.
        \begin{figure}[H]
            \centering
            \includegraphics[height=5.5cm]{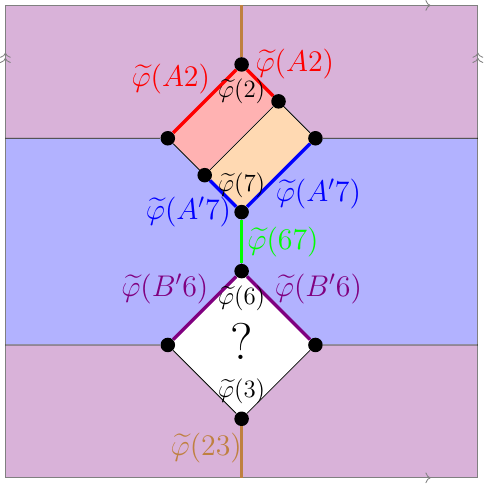}
            \caption{Possible placements of edges $A2, A'7,$ $B'6, 23, 67$}
            \label{fig:H1-H4NotTorus}
        \end{figure}        
    \noindent We observe that $\widetilde\varphi(\textcolor{red}{A2})$ is in one of two places as shown in Figure \ref{fig:H1-H4NotTorus}. Similarly, $\widetilde\varphi(\textcolor{blue}{A'7})$ is one of two places. It is now easy to see that in none of the four possible combinations does there exist a common face whose boundary cycle contains $\widetilde\varphi(\textcolor{red}{A2})$ and $\widetilde\varphi(\textcolor{blue}{A'7})$. Thus, $\widetilde\varphi$ does not extend to an embedding of $H_1$. 

    We have indicated the possible placements of the edges involved in building $H_2, H_3, H_4$ in Figure \ref{fig:H1-H4NotTorus}. Similar to the discussion about $H_1$, we quickly see that it is impossible to add the edges required to build $H_2, H_3$ or $H_4$ to the embedding in Figure \ref{fig:H1-H4NotTorus}, finishing the proof.
\end{proof}

We now show that $H_5, H_6, \ldots, H_9$ do not embed into the torus.
Recall that each of these graphs contains a subgraph homeomorphic to $G_1$.
Moreover, by \cite{Kramer24} there are exactly two inequivalent embeddings of $G_{1}$ into the torus.
\begin{figure}[H]
    \centering
    \includegraphics[height=3cm]{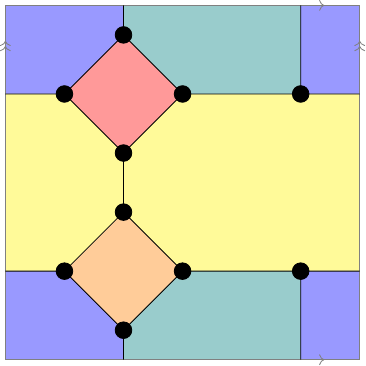} \hspace{.1cm}
    \includegraphics[height=3cm]{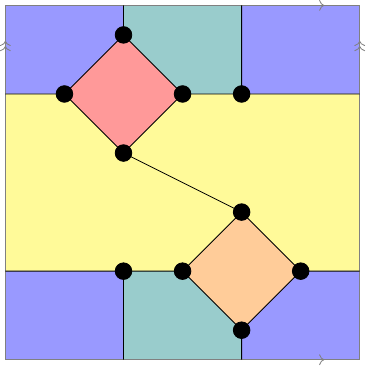}
    \caption{The two unlabelled embeddings $\varphi$ and $\psi$ of $G_1$ into the torus}
    \label{fig:G1InTorusUnabelled}
\end{figure}
\noindent To examine the choices for labelled embeddings of $G_{1}$ into the torus, we establish the following three lemmas.

\begin{lemma}
    \label{lem:LabelledStackedG1}
    Given a labelled embedding $\Phi: G_1 \hookrightarrow T^2$ that has underlying unlabelled embedding $\varphi$ as pictured on the left in Figure \ref{fig:G1InTorusUnabelled}, there exists a homeomorphism of the torus $h: T^2 \longrightarrow T^2$ such that $\widetilde\Phi:= h \circ \Phi$ is one of the partially labelled embeddings pictured in Figure \ref{fig:StackedG1Labelled}.
    \begin{figure}[H]
        \centering
        \includegraphics[width=\textwidth]{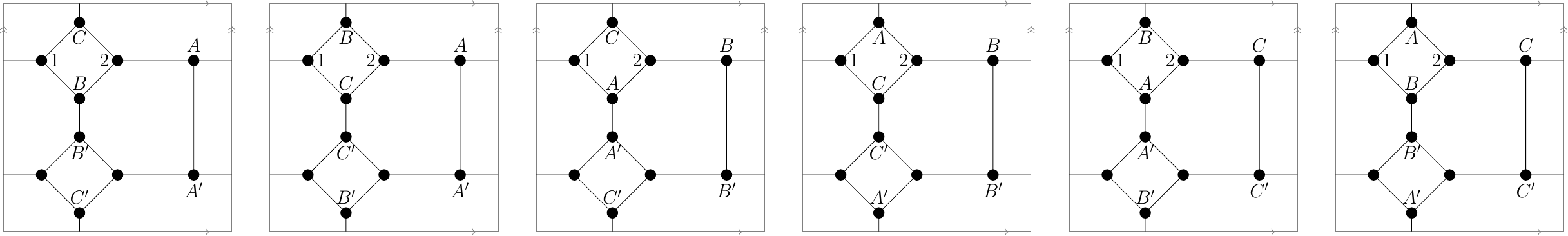}
        \caption{Partially labelled embeddings $\widetilde\Phi_i$ of $G_1$ into the torus}
        \label{fig:StackedG1Labelled}
    \end{figure}
\end{lemma}

\begin{proof}
        First note that by rotating the torus when thought of as a 3-dimensional donut, we can transform the picture on the left in Figure \ref{fig:G1InTorusUnabelled} into the picture in Figure \ref{fig:StackedSymmetric}. Notice in Figure \ref{fig:StackedSymmetric} that one of the edges of $G_1$ now lies on the boundary of the square.
        \begin{figure}[H]
            \centering
            \includegraphics[height=3cm]{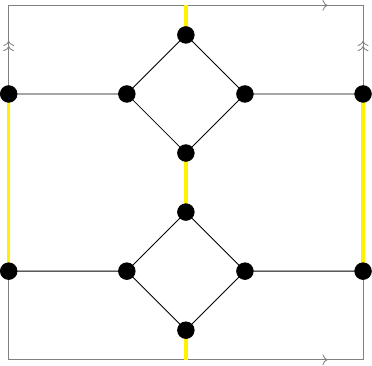}
            \caption{A more symmetric drawing of the embedding $\varphi: G_1 \hookrightarrow T^2$}
            \label{fig:StackedSymmetric}
        \end{figure}
        Recall the structure of $G_1$: The bridges $AA', BB', CC'$ are not part of any $4$-cycle in $G_1$, and the shortest path between two endpoints of distinct bridges has length two. Thus, the edges highlighted in Figure \ref{fig:StackedSymmetric} are the bridges. 
            \begin{figure}[H]
                \centering
                \includegraphics[height=3cm]{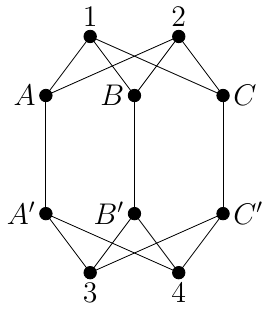}
            \end{figure}

        Consequently, the image of Vertex 1 has to be in one of four places. However, using symmetries of the torus, we may assume the image of Vertex 1 to be in the left corner of the top facial $4$-cycle.
        Note that since any $4$-cycle containing the Vertex $1$ also contains Vertex $2$ at distance two, fixing the image of Vertex $1$ also fixes the image of Vertex $2$.
        Hence, we see that a labelled embedding of $G_1$ is of the form as pictured in Figure \ref{fig:StackedG1Labelled}, where the unlabelled vertices are the images of Vertices $3$ and $4$. In particular, each partially labelled embedding pictured in Figure \ref{fig:StackedG1Labelled} gives two completely labelled embeddings.
\end{proof}

\begin{lemma}
    \label{lem:ShiftedG1Hexagonal}
    The unlabelled embedding $\widetilde\psi: G_1 \hookrightarrow T^2$ in Figure \ref{fig:ShiftedDiamondsHex} is equivalent to the embedding $\psi: G_1 \hookrightarrow T^2$ on the right in Figure \ref{fig:G1InTorusUnabelled}.
    \begin{figure}[H]
        \centering
        \includegraphics[height=4cm]{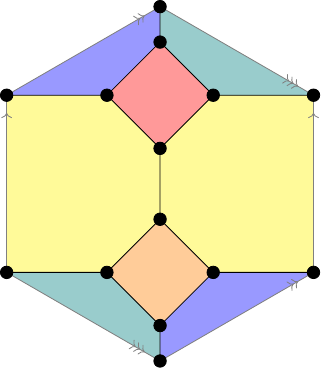}
        \caption{The embedding $\psi:G_1 \hookrightarrow T^2$ drawn on a hexagonal representation of the torus}
        \label{fig:ShiftedDiamondsHex}
    \end{figure}
\end{lemma}

\begin{proof}
        First note that the hexagonal representation of the torus can be transformed into the standard rectangular representation as shown below.
        \begin{figure}[H]
            \centering
            \includegraphics[width=.85\textwidth]{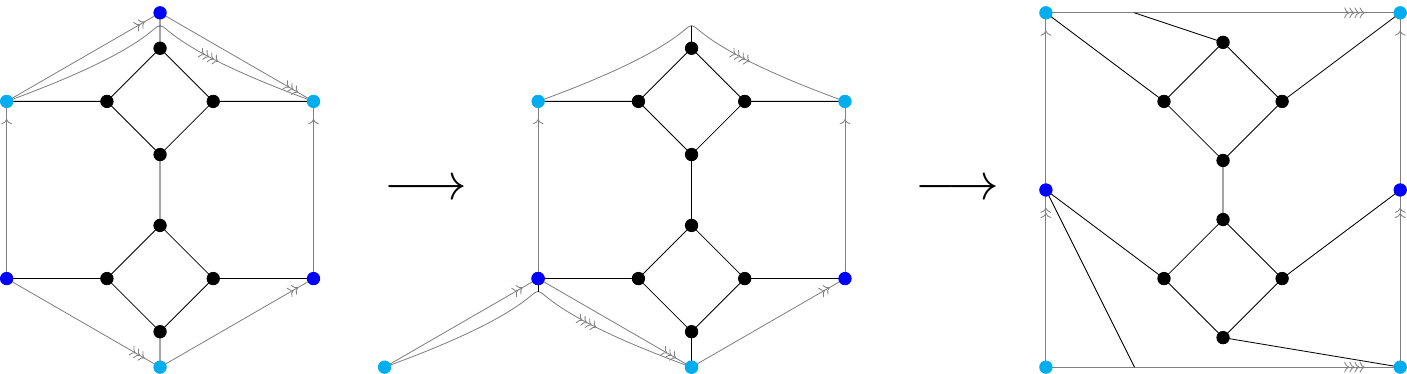}
        \end{figure}
        By rotating the torus and deforming portions of the embedding inside of a region homeomorphic to a disk, we obtain the claim.
        \begin{figure}[H]
            \centering
            \includegraphics[width=.85\textwidth]{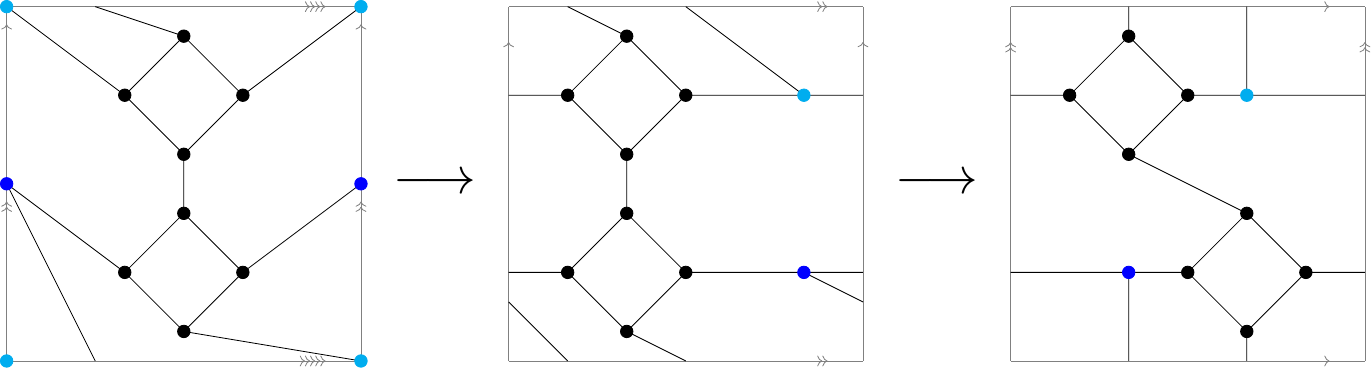}
        \end{figure}
        \vspace{-.5cm}
\end{proof}

\begin{lemma}
    \label{lem:LabelledShiftedG1}
    Given a labelled embedding $\Psi: G_1 \hookrightarrow T^2$ that has underlying unlabelled embedding $\widetilde\psi$ as pictured in Figure \ref{fig:ShiftedDiamondsHex}, there exists a homeomorphism of the torus $h: T^2 \longrightarrow T^2$ such that $\widetilde\Psi:= h \circ \Psi$ is one of the six partially labelled embeddings in Figure \ref{fig:ShiftedG1Labelled}.
    \begin{figure}[H]
        \centering
        \includegraphics[width=\textwidth]{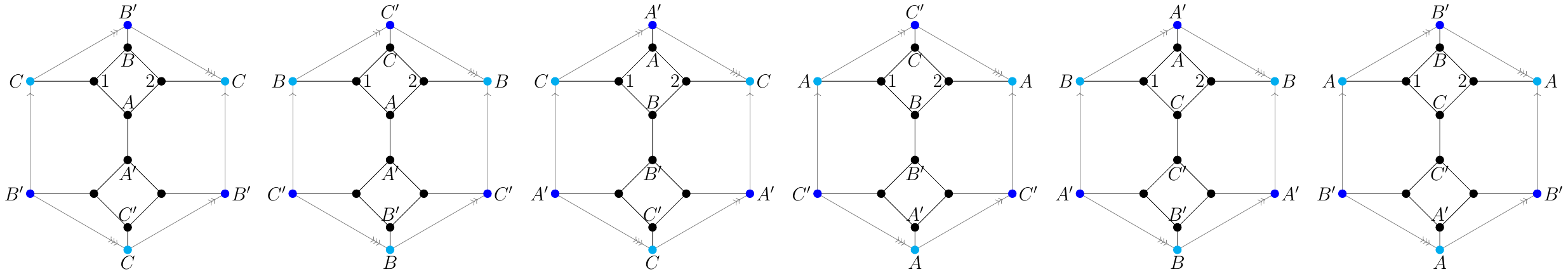}
        \caption{Partially labelled embeddings $\widetilde\Psi_i$ of $G_1$ into the torus}
        \label{fig:ShiftedG1Labelled}
    \end{figure}
\end{lemma}

\begin{proof}
        \noindent Similarly to the proof of Lemma \ref{lem:LabelledStackedG1}, by the structure of $G_1$, the edges highlighted in Figure \ref{fig:G1ShiftedHighlighted} are the bridges.
        \begin{figure}[H]
            \centering
            \includegraphics[height=3.75cm]{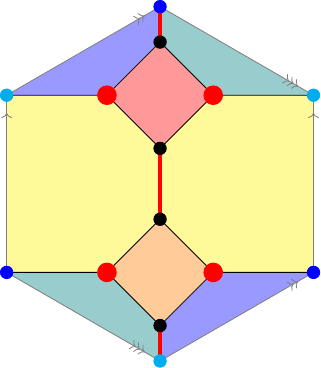}
            \caption{Embedding of $G_1$ into the torus with bridges highlighted}
            \label{fig:G1ShiftedHighlighted}
        \end{figure}
        \noindent Consequently, the image of Vertex $1$ has to be in one of the four highlighted places.
        Using the obvious action of $\ZZ_2 \times \ZZ_2$ that preserves the unlabelled image of $G_1$, we may assume the image of Vertex $1$ to be in the left corner of the top facial $4$-cycle.
        Note that since any $4$-cycle containing the Vertex $1$ also contains the Vertex $2$ at distance two, fixing the image of Vertex $1$ also fixes the image of Vertex $2$.
        Hence, we see that a labelled embedding of $G_1$ is of the form as pictured in Figure \ref{fig:ShiftedG1Labelled}, where the unlabelled vertices are the images of Vertices $3$ and $4$. In particular, each partially labelled embedding pictured in Figure \ref{fig:ShiftedG1Labelled} gives rise to two completely labelled embeddings.
\end{proof}

\begin{prop}
    \label{prop:H5-H9NotInTorus}
    $H_i$ is not toroidal for $i\in\{5, \ldots, 9\}$.
\end{prop}
\begin{proof}
    By Lemmas \ref{lem:LabelledStackedG1} and \ref{lem:LabelledShiftedG1}, we know how the labelled embeddings of $G_1$ into the torus look.
    Recall first that $H_5 \cong G_1 \cup \textcolor{red}{(AA'-BB')} \cup (\textcolor{orange}{A1}-\textcolor{teal}{B2})$ and $H_9 \cong G_1 \cup \textcolor{red}{(AA'-BB')} \cup (\textcolor{red}{xy}-\textcolor{blue}{CC'})$. Note that adding the edge $AA'-BB'$ is only possible in the first and third picture of Figure \ref{fig:StackedG1Labelled} and the last two pictures of Figure \ref{fig:ShiftedG1Labelled}. We have depicted this in Figure \ref{fig:EmbeddingsG1Bridges} below. It is now easy to see that it is impossible to add the second edge to extend any of these embeddings to an embedding of $H_5$ or $H_9$.
    \begin{figure}[H]
        \centering
        \includegraphics[width=.85\textwidth]{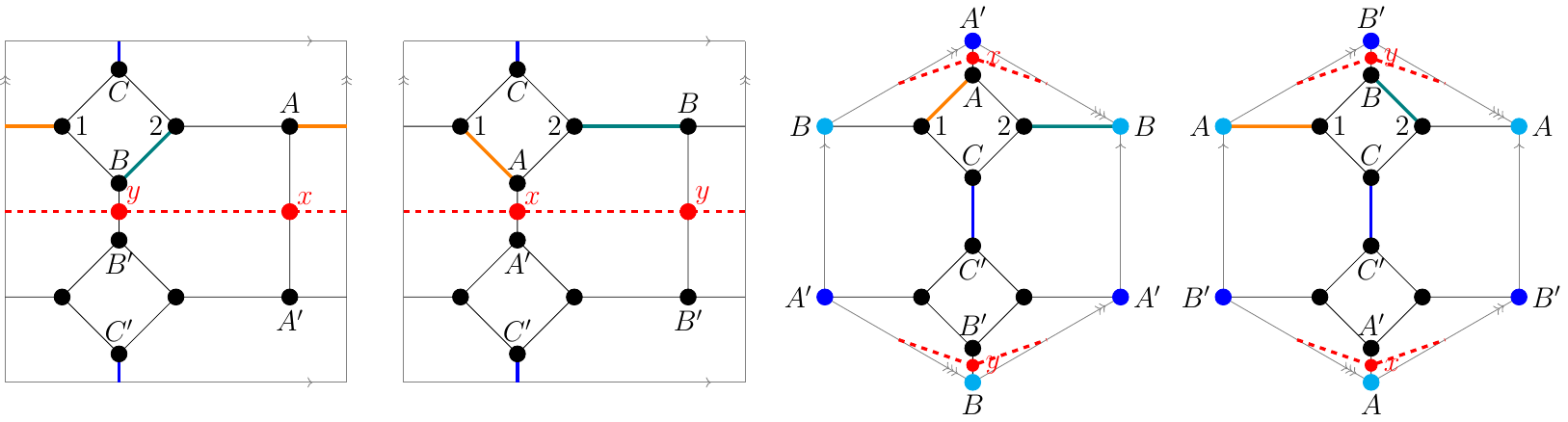}
        \caption{Labelled embeddings of $G_1 \cup (AA'-BB')$ into the torus}
        \label{fig:EmbeddingsG1Bridges}
    \end{figure}

    Next, recall that $H_6 \cong G_1 \cup \textcolor{red}{(A1-B2)} \cup (\textcolor{orange}{Ax}-\textcolor{teal}{By})$, $H_7 \cong G_1 \cup \textcolor{red}{(A1-B2)} \cup (\textcolor{orange}{Ax}-\textcolor{blue}{B1})$, and $H_8 \cong G_1 \cup \textcolor{red}{(A1-B2)} \cup (\textcolor{violet}{A2}-\textcolor{blue}{B1})$. Note that adding the edge $A1-B2$ is possible in the second, fourth, fifth and sixth picture of Figure \ref{fig:StackedG1Labelled}, and the first four pictures of Figure \ref{fig:ShiftedG1Labelled}. These are shown in Figure \ref{fig:EmbeddingsG1A1B2}. We quickly see that none of these partially labelled embeddings can be extended to an embedding of $H_6, H_7$ or $H_8$.
    \begin{figure}[H]
        \centering
        \includegraphics[width=.83\textwidth]{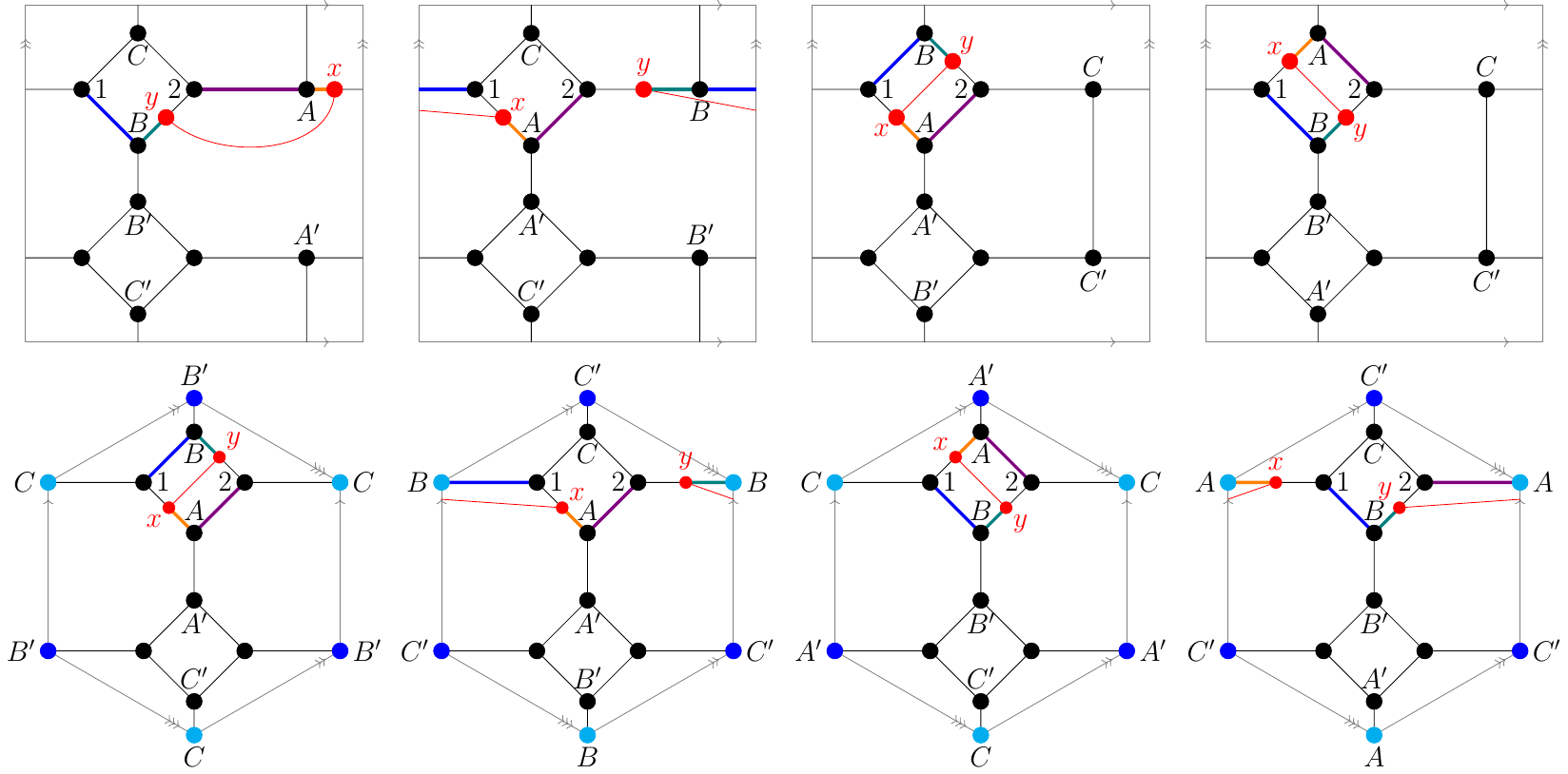}
        \caption{Labelled embeddings of $G_1 \cup (A1-B2)$ into the torus}
        \label{fig:EmbeddingsG1A1B2}
    \end{figure}
    \vspace{-1cm}
\end{proof}
% ------------------------------------------------------------------------------------------------
% ------------------------------------------------------------------------------------------------

% ------------------------------------------------------------------------------------------------
% ------------------------------------------------------------------------------------------------
\bibliographystyle{alpha}
\bibliography{main}

\end{document}